\documentclass[12pt,leqno]{article}

\usepackage{latexsym,amsmath,amscd,amssymb,framed,graphics,mathrsfs}
\usepackage{enumerate}

\usepackage{graphicx}

\usepackage[colorlinks]{hyperref}
\usepackage[all]{xy}

\makeatletter

\newcommand \al{\alpha}
\newcommand\be{\beta}
\newcommand\ga{\gamma}
\newcommand\de{\delta}

\newcommand\ze{\zeta}
\newcommand\et{\eta}
\renewcommand\th{\theta}

\newcommand\ka{\kappa}
\newcommand\la{\lambda}

\newcommand\si{\sigma}

\newcommand\ph{\varphi}

\newcommand\ps{\psi}
\newcommand\om{\omega}
\newcommand\Ga{\Gamma}

\newcommand\Th{\Theta}

\newcommand\Om{\Omega}

\newcommand\resp{resp.\ }
\newcommand\ie{i.e.\ }

\newcommand\oo{{\infty}}

\renewcommand\o{\circ}

\newcommand\x{\times}
\newcommand\on{\operatorname}
\newcommand\Ad{\on{Ad}}
\newcommand\ad{\on{ad}}\newcommand\id{\on{id}}

\newcommand\Emb{\on{Emb}}
\newcommand\Jac{\on{Jac}}

\newcommand\Den{\on{Den}}
\newcommand\Aut{\mathcal{A}ut}
\newcommand\F{\mathcal{F}}
\newcommand\Diff{\on{Diff}}

\newcommand\vol{\on{vol}}
\newcommand\ex{\on{ex}}

\newcommand\symp{\on{symp}}
\newcommand\ham{\on{ham}}

\newcommand\g{\mathfrak g}

\newcommand\h{\mathfrak h}

\newcommand\ou{\mathfrak o}
\newcommand\dd{{\bf d}}
\newcommand\QQ{\mathbf{Q}}
\newcommand\RR{\mathbb{R}}
\newcommand\PP{\mathbf{P}}
\newcommand\J{\mathbf{J}}
\newcommand\X{\mathfrak X}

\renewcommand\O{\mathcal O}
\newcommand\A{\mathcal A}
\newenvironment{proof}[1][Proof]{\noindent\textbf{#1.} }{\ \rule{0.5em}{0.5em}}

\begin{document}

\newtheorem{theorem}{Theorem}[section]
\newtheorem{definition}[theorem]{Definition}
\newtheorem{lemma}[theorem]{Lemma}
\newtheorem{remark}[theorem]{Remark}
\newtheorem{proposition}[theorem]{Proposition}
\newtheorem{corollary}[theorem]{Corollary}
\newtheorem{example}[theorem]{Example}




\title{Dual pairs for non-abelian fluids}
\author{Fran\c{c}ois Gay-Balmaz$^{1}$ and Cornelia Vizman$^{2}$ }

\addtocounter{footnote}{1}
\footnotetext{CNRS, LMD, \'Ecole Normale Sup\'erieure de Paris, France.
\texttt{francois.gay-balmaz@lmd.ens.fr }
\addtocounter{footnote}{1}}

\footnotetext{Department of Mathematics,
West University of Timi\c soara. RO--300223 Timi\c soara. Romania.
\texttt{vizman@math.uvt.ro}
\addtocounter{footnote}{1} }

\date{ }
\maketitle

\makeatother
\maketitle



\vspace{-.7cm}
\begin{abstract}
This paper is a rigorous study of two dual pairs of momentum maps arising in the context of fluid equations whose configuration Lie group is the group of automorphism of a trivial principal bundle, generically called here non-abelian fluids. It is shown that the actions involved are mutually completely orthogonal, which directly implies the dual pair property.
\end{abstract}



\section{Introduction}

\color{black}
It is well-known that the flow of the Euler equations of a perfect fluid can be formally interpreted as a geodesic on the group of volume preserving diffeomorphisms of the fluid domain, relative to an $L^2$ Riemannian metric, \cite{Arnold1966}. This result has been at the origin of many developments of the methods of symmetry and reduction for the study of incompressible fluids, their Clebsch variables and vortices, as initiated in \cite{MaWe83}. For instance, in \cite{MaWe83} J. E. Marsden and A. Weinstein discovered a pair of momentum maps associated to the Euler equations that geometrically justifies the existence of Clebsch canonical variables for ideal fluid motion and explains the Hamiltonian structure of point vortex solutions in terms of the (Lie-Poisson) Hamiltonian structure of the Euler equations. As claimed in \cite{MaWe83}, and rigorously shown in \cite{GBVi2011}, this pair of momentum maps forms a \textit{dual pair} in the sense that the Lie group associated to each of the momentum maps acts transitively on the level set of the other momentum map. In order to obtain such a result, it has been necessary to restrict each of the acting Lie groups to its commutator subgroup and to centrally extend these subgroups. This gives rise to the quantomorphism group as a central extension of the group of Hamiltonian diffeomorphisms, and to the Ismagilov central extension of the group of exact volume preserving diffeomorphisms.

A pair of momentum maps has also been described for PDEs associated to geodesics on the group of (non necessarily volume preserving) diffeomorphisms, the so-called EPDiff equations, by D. D. Holm  and J. E. Marsden in \cite{HoMa2004}. One of the momentum maps provides singular solutions of the PDE (e.g. for the $H^1$ metric, as it does for the peakon solutions of the Camassa-Holm equations \cite{CaHo1993}), whereas the other momentum map provides a constant of motion for the collective dynamics of these singular solutions, by the Noether Theorem. It has been shown in \cite{GBVi2011} that these momentum maps also form a dual pair.

Since the seminal result of Arnold, the geometric formulation via diffeomorphism groups and the associated methods of symmetry and reduction have been developed in order to apply to a large class of equations arising in hydrodynamics, such as compressible fluid and magnetohydrodynamics in \cite{MaRaWe1984}.
More recently, it has been observed (\cite{GBRa2008a}, \cite{GBRa2009}, \cite{GBRa2011a}) that several models of complex fluids and (possibly non-abelian Yang-Mills) charged fluids require the use of the group of automorphisms of a principal bundle as configuration Lie group (or its volume preserving version), instead of a group of diffeomorphisms. In this case, the base of the principal bundle is the fluid domain, whereas its structure group is given by the order parameter group in the case of complex fluids, or by the symmetry Lie group of the underlying Yang-Mills theory in the case of charged fluids. We refer to Theorem 3.3 in \cite{GBRa2011a} for a summary of the geometric formulation of several models of non-abelian fluids. We refer to \cite{BiJaLiNaPi2003}, \cite{JaNaPiPo2004}, \cite{BaMaMu2006} and reference therein, for further information about the physics literature on non-abelian fluids.

The simplest situation, which is also the case of interest for this paper, corresponds to geodesic equations on the automorphism group of a principal bundle (referred to as EPAut equations) or on its volume preserving subgroup (referred to as EPAut$_{\rm vol}$ equations). Examples are provided by the equations of motion of a non-abelian charged perfect fluid moving under the influence of a fixed external Yang-Mills field, which describe the geodesic flow of the Kaluza-Klein $L^2$ metric on the group of volume preserving automorphisms (\cite{GBRa2008a}). As explained in \cite{GBTrVi2013}, other examples treated in the literature can be seen as geodesic equations on the automorphism group of a trivial principal bundle. These are the two-component Camassa-Holm equations \cite{ChLiZh2005,Ku2007}, its modified version considered in \cite{HoOnTr2009} and its higher dimensional and anisotropic versions studied in \cite{HoTr2008}. In \cite{HoOnTr2009}, \cite{HoTr2008}, the modified versions were also shown to admit singular solutions given by a momentum map.

A pair of momentum maps explaining geometrically the above singular solutions has been considered in \cite{GBTrVi2013} in the general context of the EPAut equations on arbitrary principal bundles. In particular, it was found that these momentum maps are associated to the actions of two groups of automorphisms on a manifold of equivariant embeddings, thereby extending, from the EPDiff to the EPAut case, the momentum map setting developed in \cite{HoMa2004}.

Concerning the volume preserving situation (i.e. the EPAut$_{\rm vol}$ case), a pair of momentum maps has been found in \cite{GBTrVi2013} that extends the dual pair for the ideal fluid found in \cite{MaWe83}. Interestingly, the proper definition of one of the momentum maps needs the introduction of new infinite-dimensional Lie groups, such as the group of special Hamiltonian automorphisms and the group of Vlasov chromomorphisms, defined in \cite{GBTrVi2013}.
{In this case, one of the momentum maps yields a Clebsch representation that extends the classical Clebsch representation for ideal fluid to the non-abelian case, whereas the other momentum map recovers, as a particular case, the expression of the Klimontovich particle solution of the Yang-Mills Vlasov equations.}

\paragraph{Dual pairs of momentum maps.} Let us briefly recall the definition of a dual pair in the finite dimensional context, as formalized in \cite{We83}. Let $(M,\omega)$ be a finite dimensional symplectic manifold and let $P_1, P_2$ be two finite dimensional Poisson manifolds. A pair of Poisson mappings
\begin{equation*}
P_1\stackrel{\J_L}{\longleftarrow}(M,\om)\stackrel{\J_R}{\longrightarrow} P_2
\end{equation*}
is called a dual pair if $\ker T\J_L$ and $\ker T\J_R$ are symplectic orthogonal complements of one another, i.e.
$(\ker T\J_L)^\om=\ker T\J_R$.
In infinite dimensions, due to the weakness of the symplectic form, one has to impose both identities \cite{GBVi2011}
\begin{equation}\label{strong_dual_pair}
(\ker T\J_L)^\om=\ker T\J_R\quad\text{and}\quad (\ker T\J_R)^\om=\ker T\J_L.
\end{equation}

In many cases of interest, and this will be the case in the present paper too, the Poisson maps $\J_1$ and $\J_2$ are momentum mappings 
arising from the commuting Hamiltonian actions of two Lie groups  $H$ and $G$ on $M$. 
We assume that both momentum maps are equivariant, so that they are Poisson maps with respect to the Lie-Poisson structure on the dual Lie algebras $\h^*$ and $\g^*$:
\begin{equation}\label{dp}
\h^*\stackrel{\J_L}{\longleftarrow}(M,\om)\stackrel{\J_R}{\longrightarrow}\g^*.
\end{equation}
In this case the dual pair conditions \eqref{strong_dual_pair} become
$\mathfrak{g}_M^\om=\mathfrak{h}_M^{\om\om}$ and $\mathfrak{h}_M^\om=\mathfrak{g}_M^{\om\om}$, 
because $\ker T\J_R=\g_M^\om$ and $\ker T\J_L=\h_M^\om$.

The actions are said to be {\it mutually completely orthogonal} \cite{Libermann-Marle} if 
the $G$-orbits  and the $H$-orbits are symplectic orthogonal to each other.
In infinite dimensions we need again two identities
\begin{equation}\label{mutcomort}
\mathfrak{g}_M=\mathfrak{h}_M^{\om}\quad\text{and}\quad \mathfrak{h}_M=\mathfrak{g}_M^{\om}.
\end{equation}
These can be rewritten as $\g_M=\ker T\J_L$
and $\h_M=\ker T\J_R$, 
which means that the infinitesimal actions of $\g$ \resp $\h$ 
on level sets of momentum maps $\J_L$ \resp $\J_R$ are transitive.

In finite dimensions the mutually completely orthogonality identities \eqref{mutcomort} are equivalent to the fact that \eqref{dp} is a dual pair. {In the infinite dimensional case, the mutually complete orthogonality property \eqref{mutcomort} implies the  dual pair property, but the converse is not true. A counterexample is provided by the dual pair associated to perfect free boundary fluids, as shown in \cite{GBV2012}}.
In fact, both dual pair conditions \eqref{strong_dual_pair} 
follow from just one of the conditions \eqref{mutcomort}.

Dual pair structures arise naturally in classical mechanics. For example, in \cite{Ma1987} (see also \cite{CuRo1982}, \cite{GoSt1987} 
and \cite{Iw1985}) it was shown that the concept of dual pair of momentum maps can be useful for the study of
bifurcations in Hamiltonian systems with symmetry.
More recently, it was shown in \cite{GBVi2013} how the rigorous dual pair property of the ideal fluid momentum maps can be used to describe a new class of infinite dimensional coadjoint orbits of the Hamiltonian group.


\paragraph{Plan of the paper.} The goal of the present paper is to show that the pairs of momentum maps found in \cite{GBTrVi2013} in the context of the EPAut and EPAut$_{\rm vol}$ equations, form two \textit{dual pairs}. In the infinite dimensional situation, this means that both the equalities $(\ker T\J_1)^\omega= \ker T\J_2$ and $(\ker T\J_2)^\omega= \ker T\J_1$ are true (\cite{GBVi2011}), since the equivalence between these equalities no longer holds in the infinite dimensional situation. In fact, we will show the stronger result that the actions are mutually completely orthogonal, in the sense that the orbits of one action are symplectic orthogonal to the orbits of the other action, \textit{and vice versa}.

The plan of the paper is the following. In Section \ref{sec_2}, we recall the expression of the Euler-Poincar\'e equations on the automorphism group of a principal bundle (the EPAut equations) in the case when the principal bundle is trivial. Then after reviewing the pair of momentum maps associated to the EPDiff equations (\cite{HoMa2004}), we will recall some facts concerning the pair of momentum maps associated to the EPAut equations (\cite{GBTrVi2013}).
These momentum maps will be shown in Section \ref{sec_3} to arise from mutually completely orthogonal actions and, therefore, to form a dual pair.
In Section \ref{secti}, we recall the expression of the Euler-Poincar\'e equations on the group of volume preserving automorphisms of a trivial principal bundle (the EPAut$_{\rm\,vol}$ equations) and review from \cite{GBTrVi2013} some facts about the associated pair of momentum maps. We will then focus on a particular case relevant for the Yang-Mills Vlasov equations, arising when the total space of one of the principal bundles is a cotangent bundle. Finally, in Section \ref{sec_5} we show, still in this particular case and when the bundles are trivial, that the pair of momentum maps associated to the EPAut$_{\rm\,vol}$ arise from mutually completely orthogonal actions and is therefore a dual pair.

\medskip

\section{The EPAut equations and momentum maps}\label{sec_2}

In this Section we recall the expression of the Euler-Poincar\'e equations on the automorphism group of a principal bundle (the EPAut equations) in the case when the principal bundle is trivial, and we review some facts concerning the associated pair of momentum maps.

\subsection{The automorphism group and Euler-Poincar\'e equations}

The automorphism group $\mathcal{A}ut(P)$ of a right principal bundle $\pi: P\rightarrow M$, with structure group
$\mathcal{O}$, consists of all $\mathcal{O}$-equivariant smooth diffeomorphisms of $P$.
Any automorphism $\tilde\ph\in \mathcal{A}ut(P)$ induces a smooth diffeomorphism $\ph$ of the base $M$, by the condition 
$\pi\circ\tilde\ph=\ph\circ\pi$. The diffeomorphisms of $M$ that can be obtained this way form the subgroup $ \operatorname{Diff}(M)_{[P]}\subset \operatorname{Diff}(M)$ of diffeomorphisms preserving the isomorphism class of $P$.
It is a subgroup which contains the identity component of $\Diff(M)$,
so it consists of connected components of $\Diff(M)$.
We will denote by $\mathfrak{aut}(P)$  the Lie algebra of $ \mathcal{A}ut(P)$. It consists of all $\mathcal{O}$-equivariant smooth vector fields on $P$.

When the bundle $P$ is trivial, i.e. $P\simeq M\times \mathcal{O}\rightarrow M$, the group of all
automorphisms of $P$ is isomorphic to the semidirect product group
\[
\mathcal{A}ut(P)\simeq\operatorname{Diff}(M)\,\circledS\,\mathcal{F}(M,\mathcal{O}),
\]
where $\mathcal{F}(M,\mathcal{O})$ denotes the group of smooth $\mathcal{O}$-valued functions defined on $M$. Let us recall that the group structure of the semidirect product reads 
$$( \varphi _1 , a_1 )(\varphi _2 , a _2 )=( \varphi _1 \circ \varphi _2 , (a _1 \circ \varphi _2 ) a _2 ).$$
To a couple $(\ph,a)\in \operatorname{Diff}(M)\,\circledS\,\mathcal{F}(M,\mathcal{O})$, is associated the automorphism
\begin{equation}\label{sdir}
(x,g)\in M\times\mathcal{O}\mapsto (\ph(x),a(x)g)\in M\times\mathcal{O}.
\end{equation}
The Lie algebra of the automorphism group is isomorphic to the semidirect product Lie algebra
$\mathfrak{aut}(P)\simeq\mathfrak{X}(M)\,\circledS\,\mathcal{F}(M,\mathfrak{o})$,
where $\mathfrak{o}$ is the Lie algebra of the structure group
$\mathcal{O}$, $\mathfrak{X}(M)$ denotes the space of smooth vector fields on $M$, and $\mathcal{F}(M,\mathfrak{o})$ denotes the Lie algebra of smooth $\mathfrak{o}$-valued functions defined on $M$.

\color{black}

\paragraph{Euler-Poincar\'e equations, EPDiff, and EPAut.} Let $G$ be a Lie group with Lie algebra $ \mathfrak{g}  $ and consider a right $G$-invariant Lagrangian $L: TG \rightarrow \mathbb{R}  $ defined on the tangent bundle $TG$ of $G$. Let $\ell: \mathfrak{g}  \rightarrow \mathbb{R}  $ be the reduced Lagrangian associated to $L$, that is $\ell( \xi )=L(g, \dot g)$, with $\xi = \dot g g ^{-1} $. By applying the process of Lagrangian reduction, the Euler-Lagrange equations for $L$ are equivalent to the Euler-Poincar\'e equations for $\ell$, 
\[
\partial _t \frac{\delta \ell}{\delta \xi }+ \operatorname{ad}^*_ \xi \frac{\delta \ell}{\delta \xi }=0,
\] 
see e.g. \cite{MaRa99} for a detailed exposition. When $G=\operatorname{Diff}(M)$, these equations are called the EPDiff equations, see \cite{HoMa2004}. Similarly, when $G= \mathcal{A}ut(P)$, these equations will be called the \textit{EPAut equations}.
In the following Proposition we give the EPAut equations in the case when $P$ is a trivial principal bundle. We refer to \cite{Vi2001}, \cite{Vi2008}, for the geodesic case and to \cite{GBTrVi2013} for the general case.

\begin{proposition}[The EPAut equations on a trivial bundle] {\rm \cite{GBTrVi2013}}\label{EPaut_trivial}
Consider a reduced Lagrangian  $\ell:\mathfrak{X}(M)\,\circledS\,\mathcal{F}(M,\mathfrak{o}) \to\RR$ and identify the dual Lie algebra with the space $\Om^1(M)\otimes\Den(M)\x\F(M,\ou^*)\otimes\Den(M)$ by using the $L^2$ pairing, where $\Den(M)$ denotes the space of densities on $M$. Then the EPAut equations are
\begin{equation}\label{EPAut_trivial}
\left\{\begin{array}{l}
\vspace{0.2cm}\displaystyle\frac{\partial}{\partial t} \frac{\delta \ell}{\delta\mathbf{u}} +\pounds_{\mathbf{u}}\frac{\delta \ell}{\delta\mathbf{u}}+\frac{\delta \ell}{\delta{\nu}}\!\cdot\!\mathbf{d}{\nu}=0\\
\displaystyle\frac{\partial}{\partial t}\frac{\delta \ell}{\delta{\nu}}+\pounds_{\mathbf{u}}\frac{\delta \ell}{\delta{\nu}}+\operatorname{ad}^*_{{\nu}}\frac{\delta \ell}{\delta{\nu}}=0,
\end{array}\right.
\end{equation}
where $(\mathbf{u},{\nu})\in\mathfrak{X}(M)\,\circledS\,\mathcal{F}(M,\mathfrak{o})$ and the operator $\pounds_\mathbf{u}$ denotes the Lie derivative acting on tensor densities.
\end{proposition}

In the special case $M=S^1$ and $G=S^1$, for suitable Lagrangians, we obtain 
the two component Camassa-Holm equation \cite{ChLiZh2005} 
and the modified two-component Camassa-Holm equation \cite{HoOnTr2009}. \textcolor{black}{Equations for complex and nonabelian fluids are obtained from the Euler-Poincar\'e equations \eqref{EPAut_trivial} by extending them to include advected quantities or/and coupling them with the Euler-Lagrange equations for the Yang-Mills fields. We refer to \cite{GiHoKu1983}, \cite{Ho1986}, \cite{HoKu1988}, \cite{Ho2002} for the description of the noncanonical Hamiltonian structures and to \cite{GBRa2008a}, \cite{GBRa2009}, \cite{GBRa2011a} for the Lagrangian and Hamiltonian reductions approaches.}

\begin{remark}[Manifold structures]\label{smooth}{\rm In this paper, all the (finite dimensional) manifolds involved are smooth, Hausdorff, and paracompact (to admit partition of unity).
All the maps considered are smooth (i.e. $C^\infty$). All the manifolds are assumed to have no boundary.

The space of smooth functions defined on a compact manifold is a Fr\'echet manifold
in a natural way. The space of embeddings is an open subset 
of this Fr\'echet manifold, hence a Fr\'echet manifold itself \cite{KM}.
}
\end{remark}


\subsection{A pair of momentum maps for the EPAut equations}

\color{black}
\paragraph{Review of the EPDiff case.} Let $S$ and $M$ be two manifolds with $ \operatorname{dim}S \leq \operatorname{dim}M$. Suppose that $S$ {is compact} and carries a volume form $ \mu _S $. 
As explained in  Remark \ref{smooth}, the space $ \operatorname{Emb}(S,M)$ of embeddings of $S$ into $M$ is a Fr\'echet manifold.
Recall from \cite{HoMa2004} that the pair of momentum maps associated to the EPDiff equation is obtained by considering the cotangent lifted action of the groups $ \operatorname{Diff}(S)$ and $ \operatorname{Diff}(M)$ on the manifold $ \operatorname{Emb}(S,M)$:
\begin{equation}\label{EPDiff_dualpair} 
\mathfrak{X}(M) ^\ast \stackrel{\;\J_L}{\longleftarrow}T^* \operatorname{Emb}(S,M) \stackrel{\J_R}{\longrightarrow} \mathfrak{X}(S) ^\ast.
\end{equation}
Note that the tangent space $T_ \mathbf{Q} \operatorname{Emb}(S, M)$ consists of vector fields $\mathbf{V} _ \mathbf{Q} :S \rightarrow TM$ covering the embedding $ \mathbf{Q} $. Since a volume form $ \mu _S $ has been fixed, the cotangent space $T^*_ \mathbf{Q} \operatorname{Emb}(S,M)$ can be identified with the space of 1-forms $\mathbf{P} _ \mathbf{Q}  : S \rightarrow T^*M$ covering $ \mathbf{Q} $.

The left momentum map 
$$ \mathbf{J} _L(\mathbf{P}_ \mathbf{Q}  )=\int_S\mathbf{P}_ \mathbf{Q}  (x)\delta (m- \mathbf{Q} (x))\mu _S$$ 
provides the formula for singular solutions of the EPDiff equations, whereas the right momentum map 
$$ \mathbf{J} _R(\mathbf{P}_ \mathbf{Q} )= \mathbf{P}_ \mathbf{Q}  \cdot T \mathbf{Q}$$ 
provides a Noether conserved quantity for the (collective) Hamiltonian dynamics of these singular solutions in terms of the canonical variable $\mathbf{P}_ \mathbf{Q} \in T^* \operatorname{Emb}(S,M)$. Here $T \mathbf{Q} :TS \rightarrow TM$ denotes the tangent map to the embedding $\mathbf{Q}$.

These expressions of the momentum maps are obtained from the following general formula. Let $G$ be a Lie group acting on a manifold $Q$, and consider the cotangent lifted action of $G$ on $T^*Q$. Then this action admits the equivariant momentum map $\mathbb{J}:T^*Q\rightarrow\mathfrak{g}^*$ given by
\begin{equation}\label{ctgm}
\langle \mathbb{J}(\alpha_q),\xi\rangle=\langle\alpha_q,\xi_Q(q)\rangle,\quad \alpha_q\in T^*Q, \quad \xi\in\mathfrak{g},
\end{equation}
where $\xi_Q$ is the infinitesimal generator of the action of $G$ on $Q$ associated to the Lie algebra element $\xi$. Recall that the momentum map $ \mathbb{J}  $  verifies the condition $ \mathbf{d}  \mathbb{J}  _ \xi = \mathbf{i} _{ \xi _{T^*Q}} \Omega _{can}$, where $ \Omega _{can}$ is the canonical symplectic form, $ \xi _{T^*Q}$ is the infinitesimal generator of the $G$-action on $T^*Q$, and $ \mathbb{J}  _ \xi :T^*Q \rightarrow \mathbb{R}  $ is the function defined by $ \mathbb{J}  _ \xi ( \alpha _q ):= \left\langle \mathbb{J} ( \alpha _q ), \xi \right\rangle $. By equivariance, $ \mathbb{J} $ is a Poisson map relative to the canonical symplectic form on $T^*Q$ and the Lie-Poisson structure on $ \mathfrak{g}  ^\ast $.

\paragraph{Momentum maps for the EPAut equations.} In \cite{GBTrVi2013}, a pair of momentum maps analogue to \eqref{EPDiff_dualpair} has been constructed for the EPAut equations.
It is obtained by considering two principal $ \mathcal{O} $-bundles $\pi _S:P _S \rightarrow S$ and $\pi _M:P_M  \rightarrow M$, and the cotangent lifted action of the automorphism groups $ \mathcal{A}ut(P_S)$ and $ \mathcal{A}ut(P_M)$ on the manifold $ Q_{KK}:=\operatorname{Emb}_ \mathcal{O} (P_S, P_M)$ of all smooth $ \mathcal{O} $-equivariant embeddings of $P_S$ into $P_M$:
\begin{equation}\label{EPAut_dualpair} 
\mathfrak{aut}(P_M) ^\ast \stackrel{\J_L}{\longleftarrow}T^* \operatorname{Emb}_ \mathcal{O} (P_S,P_M) \stackrel{\J_R}{\longrightarrow} \mathfrak{aut}(P_S) ^\ast.
\end{equation}
We assume that both $S$ and $ \mathcal{O} $ are compact,
so $P_S$ is compact too and $Q_{KK}$ is a Fr\'echet manifold (see Remark \ref{smooth}).

When both $P_S$ and $P_M$ are trivial principal bundles, we have the identification
\begin{equation}\label{QKK_trivial}
Q_{KK}\simeq \operatorname{Emb}(S,M)\times\mathcal{F}(S,\mathcal{O}),
\end{equation}
see Lemma 3.7 in \cite{GBTrVi2013}. 
In a similar way with \eqref{sdir}, the equivariant embedding associated to a pair $(\mathbf{Q} , \gamma ) \in  \operatorname{Emb}(S,M)\times\mathcal{F}(S,\mathcal{O}) $ is the map $(x,g)\in S\x\O\mapsto(\QQ(x),\ga(x)g)\in M\x\O$. The tangent space $T_\gamma \mathcal{F} (S, \mathcal{O} )$ consists of functions $ v_ \gamma : S \rightarrow T \mathcal{O} $ covering $ \gamma $ and the cotangent space $T^*_ \gamma \mathcal{F} (S, \mathcal{O} )$ consists of functions $ \kappa _ \gamma :S \rightarrow T^* \mathcal{O} $ covering $ \gamma $.

We now recall from \cite{GBTrVi2013} the expression of the cotangent momentum maps in the case of trivial bundles.

\color{black}

\subsubsection{Left action momentum map} 

The left action of
$\mathcal{A}ut(M\x\O)\simeq\operatorname{Diff}(M)\,\circledS\,\mathcal{F}(M,\mathcal{O})\ni ( \varphi , a)$
on $Q_{KK}\ni ( \mathbf{Q} , \gamma )$ is defined by
\begin{equation}\label{left_action_trivial}
(\ph,a)(\mathbf{Q},\ga):=(\ph\circ\mathbf{Q},(a\circ\mathbf{Q})\ga).
\end{equation}
Given a Lie algebra element $(\mathbf{u},{\nu})\in
\mathfrak{X}(M)\,\circledS\,\mathcal{F}(M,\mathfrak{o})$, the
infinitesimal generator associated to the left action
\eqref{left_action_trivial} reads
$(\mathbf{u},{\nu})_{Q_{KK}}(\mathbf{Q},\ga)=(\mathbf{u}\circ\mathbf{Q},({\nu}\circ\mathbf{Q})\ga)$.
By applying formula \eqref{ctgm} with $\mathfrak{g}  =\mathfrak{X}(M)\,\circledS\,\mathcal{F}(M,\mathfrak{o}) $, we get the expression
\[
\langle\J_L(\PP_\QQ,\ka_\ga),(\mathbf{u},\nu)\rangle=\langle\PP_\QQ,\mathbf{u}\o\QQ\rangle+
\langle\ka_\ga,(\nu\o\QQ)\ga\rangle,
\]
so that the momentum map $\mathbf{J}_L:T^*Q_{KK}\rightarrow\mathfrak{X}(M)^*\times\mathcal{F}(M,\mathfrak{o})^*$ reads
\begin{equation} \label{left_momap_trivial}
\mathbf{J}_L\left(\mathbf{P}_\mathbf{Q},\kappa_\ga\right)=
\left(\int_S\mathbf{P}_\mathbf{Q}\delta(x-\mathbf{Q})\mu _S ,\int_S\kappa_{\ga}\ga^{-1}\delta(x-\mathbf{Q})\mu _S \right),
\end{equation} 
where $\mathfrak{X}(M)^*$ and $\mathcal{F}(M,\mathfrak{o})^*$ denote the distributional dual spaces.


\subsubsection{Right action momentum map}

The right action of
$\mathcal{A}ut(S\x\O)=\operatorname{Diff}(S)\,\circledS\,\mathcal{F}(S,\mathcal{O})\ni ( \psi , b)$
on $Q_{KK}\ni ( \mathbf{Q} , \gamma )$ is defined by
\begin{equation}\label{right_action_trivial}
(\mathbf{Q},\ga)(\ps,b):=(\mathbf{Q}\circ\ps,(\ga\circ\ps)b).
\end{equation}
Given a Lie algebra element $(\mathbf{v},\ze)\in\X(S)\,\circledS\,\F(S,\ou)$, the infinitesimal generator reads
$(\mathbf{v},\zeta)_{Q_{KK}}(\mathbf{Q},\ga)=(T\mathbf{Q}\!\cdot\!\mathbf{v},T\ga\!\cdot\!\mathbf{v}+\ga\zeta)$,
where $ \gamma \zeta \in T_ \gamma \mathcal{F} (S, \mathcal{O} )$ denotes the left translation of $ \zeta $ by $ \gamma $.
By applying the cotangent momentum map formula \eqref{ctgm} to $\g=\X(S)\,\circledS\,\F(S,\ou)$, we get
\[
\langle\J_R(\PP_\QQ,\ka_\ga),(\mathbf{v},\ze)\rangle=\langle\PP_\QQ,T\QQ \cdot \mathbf{v}\rangle
+\langle\ka_\ga,T\ga\cdot\mathbf{v}+\ga\ze\rangle,
\]
so that the momentum map $\mathbf{J}_R:T^*Q_{KK}\rightarrow\mathfrak{X}(S)^*\times\mathcal{F}(S,\mathfrak{o})^* $ reads
\begin{equation}\label{momap_right_trivial}
\mathbf{J}_R\left(\mathbf{P}_\mathbf{Q},\kappa_\ga\right)
=\left(\mathbf{P}_\mathbf{Q}\!\cdot\!T\mathbf{Q}+\kappa_\ga\!\cdot\!T\ga,\ga^{-1}\kappa_\ga\right),
\end{equation}
where we note that, as opposed to $ \mathbf{J} _L $ in \eqref{left_momap_trivial}, $ \mathbf{J} _R $ takes values in the regular dual.
 
\begin{remark}\normalfont
Both momentum maps are equivariant, since they are momentum maps for cotangent lifted actions.
In particular they are \textcolor{black}{formally} Poisson maps for the Lie-Poisson bracket on  Lie algebra duals. \textcolor{black}{Note that the definition of Poisson brackets leads to several difficulties in the infinite dimensional case, this is why the above property only holds at a formal level.} In the next section we will show that these momentum maps form a dual pair. \textcolor{black}{Even if the Poisson properties only hold at a formal level, the dual pair property can be rigorously verified, as in \cite{GBVi2011}.}
\end{remark}


\section{The dual pair property of the EPAut momentum maps}\label{sec_3}

As shown in \cite{GBVi2011}, the EPDiff momentum maps \eqref{EPDiff_dualpair} introduced in \cite{{HoMa2004}} form a dual pair when restricted to the open subset $T^*\Emb(S,M)^\times $ 
of $T^*\Emb(S,M)$, consisting of one-forms on $M$ along $S$ 
which are everywhere non-zero on $S$, \ie $\PP_\QQ(x)\ne 0$ for all $x\in S$.
Note that $T^*\Emb(S,M)^\times $ is invariant under the actions of both $\Diff(S)$ and $\Diff(M)$.

The pair of EPAut momentum maps \eqref{EPAut_dualpair} was introduced in \cite{{GBTrVi2013}}.
We shall now show the dual pair property for these EPAut momentum maps  in the case of trivial bundles, that is, with momentum maps given by \eqref{left_momap_trivial} and \eqref{momap_right_trivial}.

\subsection{Cotangent lifted actions}

Let $T^*Q_{KK}^\times$ denote the open subset of
$T^*Q_{KK}$ consisting of pairs $(\PP_\QQ,\ka_\ga)$ with the property that for any $x\in S$, the pair
$(\PP_\QQ(x),\ka_\ga(x))\in T^*_{\QQ(x)}M\x T^*_{\ga(x)}\O$ is non-zero.
Note again that $T^*Q_{KK}^\times$ is invariant under the actions of both $\Aut(S\x\O)$ and $\Aut(M\x\O)$. 

We compute below the explicit form of the cotangent lifted actions of the automorphism groups on $T^*Q_{KK}^\times$, that we will need in order to prove the dual pair property.

\paragraph{Left cotangent action.}
The formula for the cotangent lifted action of $\Aut(M\x\O)=\Diff(M)\,\circledS\, \F(M,\O)$ on $T^*Q_{KK}$
is\begin{equation}\label{left_cotan}
(\ph,a)\cdot(\PP_\QQ,\ka_\ga)=\left(T^* \varphi ^{-1} \cdot \left(\PP_\QQ
-\langle\ka_\ga\ga^{-1},\de^la\o\QQ\rangle\right),
(a\o \QQ)\ka_\ga\right),
\end{equation}
where $\de^l a=a^{-1}\dd a\in \Omega ^1  (M,\mathfrak{o} )$ denotes the left logarithmic derivative of $a\in\F(M,\O)$, 
and therefore $\de^l a\o\QQ$ is a section of the vector bundle  $\QQ^*(T^*M)\otimes\mathfrak{o}\to S$. Indeed, using  the expression of the tangent lift of the  action \eqref{left_action_trivial} given by
\[
(\ph,a)\cdot(\mathbf{V} _\QQ,v_\ga)=\left(T\ph \cdot \mathbf{V} _\QQ,(a\o\QQ)v_\ga+(Ta \cdot \mathbf{V} _\QQ)\ga\right),
\]
one computes the cotangent lifted action \eqref{left_cotan} as follows. 
Using the expression of the inverse in the semidirect product
$(\ph,a)^{-1}=(\ph^{-1},a^{-1}\o\ph^{-1})$,
we have
\begin{align*}
\langle(\ph,a)&\cdot(\PP_\QQ,\ka_\ga),(V_{\ph\o\QQ},v_{(a\o\QQ)\ga})\rangle
=\left\langle (\PP_\QQ,\ka_\ga),(\ph^{-1},a^{-1}\o\ph^{-1})\cdot(V_{\ph\o\QQ},v_{(a\o\QQ)\ga})\right\rangle \\
&=\left\langle \PP_\QQ,T\ph^{-1} \cdot V_{\ph\o\QQ}\right\rangle 
+\left\langle \ka_\ga,(a^{-1}\o\QQ)v_{(a\o\QQ)\ga}+(T(a^{-1}\o\ph^{-1}) \cdot V_{\ph\o\QQ})(a\o\QQ)\ga\right\rangle \\
&=\left\langle \PP_\QQ,T\ph^{-1} \cdot V_{\ph\o\QQ}\right\rangle 
+\left\langle \ka_\ga\ga^{-1},(Ta^{-1}(T\ph^{-1}\cdot V_{\ph\o\QQ}))(a\o\QQ),\right\rangle 
+\left\langle (a\o\QQ)\ka_\ga,v_{(a\o\QQ)\ga}\right\rangle \\
&=\left\langle T^* \varphi ^{-1} \cdot 
\left(\PP_\QQ
+\langle\ka_\ga\ga^{-1},\de^r (a^{-1})\o\QQ\rangle\right)
,V_{\ph\o\QQ}\right\rangle 
+\left\langle (a\o\QQ)\ka_\ga,v_{(a\o\QQ)\ga}\right\rangle\\
&=\left\langle T^* \varphi ^{-1} \cdot 
\left(\PP_\QQ
-\langle\ka_\ga\ga^{-1},\de^l a\o\QQ\rangle\right)
,V_{\ph\o\QQ}\right\rangle 
+\left\langle (a\o\QQ)\ka_\ga,v_{(a\o\QQ)\ga}\right\rangle,
\end{align*}
which shows \eqref{left_cotan}.
We used in the last step the relation  
$\de^r(a^{-1})=-\de^l a$ between left and right logarithmic derivatives,
where the right logarithmic derivative is defined by 
$\de^r a=(\dd a) a^{-1}\in \Omega ^1  (M,\mathfrak{o} )$.

\paragraph{Right cotangent action.} The formula for the cotangent lifted action of
$\Aut(S\x\O)=\Diff(S)\,\circledS\, \F(S,\O)$ on $T^*Q_{KK}$ reads
\begin{equation}\label{right_cotangent}
(\PP_\QQ,\ka_\ga)\cdot(\ps,b)
=\left(  (\PP_\QQ\circ \psi )\Jac_ \psi,(\ka_\ga\circ \psi)b\Jac_\ps\right),
\end{equation}
where $\Jac_ \psi $ denotes the Jacobian of $ \psi $ relative to the volume form $ \mu _S $, i.e. $ \psi ^\ast \mu _S = \Jac_ \psi \mu _S $.
Indeed, from the expression of the tangent lift of the action \eqref{right_action_trivial} given by
\[
(V_\QQ,v_\ga)\cdot (\ps,b)=(V_\QQ\o\ps,(v_\ga\o\ps)b)
\]
one computes the cotangent lifted action \eqref{right_cotangent} as follows.
\begin{align*}
\langle(\PP_\QQ,\ka_\ga)\cdot(\ps,b),(V_{\QQ\o\ps},v_{(\ga\o\ps)b})\rangle
&=\langle(\PP_\QQ,\ka_\ga),(V_{\QQ\o\ps},v_{(\ga\o\ps)b})\cdot(\ps^{-1},b^{-1}\o\ps^{-1})\rangle\\
&=\langle\PP_\QQ,V_{\QQ\o\ps}\o\ps^{-1}\rangle
+\langle\ka_\ga,(v_{(\ga\o\ps)b}b^{-1})\o\ps^{-1})\rangle\\
&=\langle(\PP_\QQ\o\ps)\Jac_\ps,V_{\QQ\o\ps}\rangle
+\langle(\ka_\ga\o\ps)b\Jac_\ps ,v_{(\ga\o\ps)b}\rangle,
\end{align*}
which shows \eqref{right_cotangent}. 

\subsection{Transitivity results}

In order to prove the dual pair property, we shall show the stronger result 
that the actions of $\Aut(M\x\O)$ and $\Aut(S\x\O)$ on $T^*Q_{KK}^\times$ are mutually completely orthogonal. 

We will use the following lemma (Lemma 4.2 in \cite{GBVi2011}), detached from the proof of Proposition 3 in \cite{HaVi2004}.

\begin{lemma}\label{function}
Let $(M,g)$ be a Riemannian manifold and $N\subset M$ a submanifold. Let $E$ be the normal bundle $TN^\perp$ 
viewed as a tubular neighborhood of $N$ in $M$,
with the identification done by the
Riemannian exponential map. 
Then any section $\la\in\Ga(T^*M|_N)$ vanishing on $TN$,
when restricted to $TN^\perp$, 
defines a smooth function $h\in\F(E)$,
linear on each fiber,
whose differential along $N$ satisfies $(\dd h)|_N=\la$
$($as sections of $T^*M|_{N}$$)$.
\end{lemma}

The first transitivity result is shown in the following proposition.

\begin{proposition}\label{transitivity1} 
The group $\Aut(M\x\O)$ acts infinitesimally transitively on the
level sets of the momentum map $\mathbf{J}_R$ given in \eqref{momap_right_trivial} restricted to $T^*Q_{KK}^\x$.
\end{proposition}

\begin{proof} The left action $(\ph,a)(\mathbf{Q},\ga)=(\ph\circ\mathbf{Q},(a\circ\mathbf{Q})\ga)$ 
of $\Aut(M\x\O)$ on $Q_{KK}$ is transitive on connected components.
This can be seen as follows. Given $(\QQ_1,\ga_1)$ and $(\QQ_2,\ga_2)$ in the same
connected component of $Q_{KK}$, by the transitivity of the action of $\Diff(M)$ 
on connected components of $\Emb(S,M)$ \cite{Hi1976}, we find 
a diffeomorphism $\ph$ such that $\QQ_2=\ph\o\QQ_1$.
By a standard argument using a partition of unity,
it is possible to extend $\ga_2\o\ga_1^{-1}\in\F(S,\O)$ in a smooth way to 
a smooth function $a\in\F(M,\O)$ along the embedding $\QQ_1$.

Consider two pairs $(\PP_\QQ,\ka_\ga)$ and $(\PP'_{\QQ'},\ka'_{\ga'})$
in the same level set of $\J_R$,
so that
\begin{equation}\label{leve}
\PP_\QQ\cdot T\QQ+\ka_\ga\cdot T\ga=\PP'_{\QQ'}\cdot T\QQ'+\ka'_{\ga'}\cdot T\ga'\quad\text{and}\quad 
\ga^{-1}\ka_\ga=(\ga')^{-1}\ka'_{\ga'}.
\end{equation}
The left action of $\Aut(M\x\O)$ being transitive on connected components of $Q_{KK}$,
we can focus on the action of the isotropy group of $(\QQ,\ga)$ on the
cotangent fiber over $(\QQ,\ga)$. 
Hence we assume that both pairs have the same foot point $(\QQ,\ga)$, so that the identities above become
\[
\PP_\QQ\cdot T\QQ=\PP'_{\QQ}\cdot T\QQ\quad\text{and}\quad 
\ka'_\ga=\ka_\ga.
\]
The two objects that characterize the $\J_R$ level set of $(\PP_\QQ,\ka_\ga)$ over $(\QQ,\ga)$
are 
\[\al:=\PP_\QQ\cdot T\QQ\in\Om^1(S)\quad\text{ and }\quad
\si:=\ka_\ga\ga^{-1}\in\F(S,\mathfrak{o}^*).
\] 
Because the embedding $\QQ$ is fixed, we can move the objects to  the submanifold $N=\QQ(S)$ of $M$.
We will denote by the same letters their push-forward by the diffeomorphism 
$\QQ:S\to N$,
namely $\al\in\Om^1(N)$ and $\si\in\F(N,\mathfrak{o}^*)$. 
We also identify $T^*_\QQ\Emb(S,M)=\Ga(\QQ^*T^*M)$ with $\Ga(T^*M|_N)$,
via the diffeomorphism $\QQ$.
We consider the affine subspace
\[
\Ga_\al(T^*M|_N):=\{\PP\in\Ga(T^*M|_N) : \PP|_{TN}=\al\}\subset \Ga(T^*M|_N),
\]
whose directing linear subspace is given by the space of sections of the conormal bundle to the submanifold $N$ of $M$
\[
\Ga_0(T^*M|_N)=\{\PP\in\Ga(T^*M|_N):\PP|_{TN}=0\}.
\]
We define an open subset of $\Ga_\al(T^*M|_N)$ by
\[
\Ga_\al(T^*M|_N)^\si:=\{\PP\in\Ga(T^*M|_N)\,:\,\PP|_{TN}=\al
\text{ and }\PP(x)\ne 0,\forall \,x\in \si^{-1}(0)\}.
\]

The isotropy group of $(\QQ,\ga)$ is a semidirect product group,
\begin{equation}\label{group_iso}
\Aut(M\x\O)_{(\QQ,\ga)}=\Diff_N(M)\,\circledS\, \F_{N}(M,\O),
\end{equation}
where $\Diff_N(M)$ denotes the subgroup of diffeomorphisms of $M$ that fix 
the submanifold $N=\QQ(S)$ of $M$ pointwise  
and $\F_{N}(M,\O):=\{a:M\to \O\,|\,a(x)=e,\forall x\in N\}$.
Its isotropy Lie algebra of $(\QQ,\ga)$ is the semidirect product 
$\X_N(M)\,\circledS\, \F_N(M,\mathfrak{o})$,
where $\X_N(M)$ denotes the Lie algebra of vector fields on $M$ that vanish on $N$,
the Lie algebra of $\Diff_N(M)$, and  $\F_{N}(M,\mathfrak{o}):=\{\nu:M\to \mathfrak{o}\mid\nu(x)=0,\forall \,x\in N\}$, the Lie algebra of $\F_N(M,\O)$.

The isotropy group acts by \eqref{left_cotan} on the cotangent fiber over $(\QQ,\ga)$, 
so it preserves the second component
\begin{equation}\label{isot}
(\ph,a)\cdot(\PP_\QQ,\ka_\ga)=\left(T^* \varphi ^{-1} \cdot \left(\PP_\QQ
-\langle\ka_\ga\ga^{-1},\de^la\o\QQ\rangle\right).
\ka_\ga\right).
\end{equation}
Its restriction 
to the level sets of $\J_R$ can be seen an action on  $\Ga_\al(T^*M|_N)^\si$:
\begin{equation}\label{isotrop_action}
(\ph,a)\cdot\PP=\left(\PP
-\langle\si,(\de^la)|_N\rangle\right)\o(T\ph^{-1})|_N.
\end{equation}
We notice that $(\de^la)|_{TN}=\de^l(a|_N)=\de^le=0$ for all $a\in\F_N(M,\O)$
and $T\ph^{-1}|_{TN}=\id_{TN}$ for all $\ph\in\Diff_N(M)$,
so we have indeed an action on $\Ga_\al(T^*M|_N)^\si$.

The infinitesimal action of the isotropy Lie algebra $\X_N(M)\,\circledS\, \F_N(M,\mathfrak{o})$ on $\Ga_\al(T^*M|_N)^\si$ 
is obtained by differentiating the action \eqref{isotrop_action}:
\begin{equation}\label{infi}
(\mathbf{u},\nu)_{\Ga_\al(T^*M|_N)^\si}(\PP)=
-\PP\o(\nabla\mathbf{u})|_N-\langle\si,(\dd\nu)|_N\rangle,
\end{equation}
where we used the fact that the differential at the identity of the logarithmic derivative map
$\de^l:\F(M,\O)\to\Om^1(M,\mathfrak{o})$ is
$\dd:\F(M,\mathfrak{o})\to\Om^1(M,\mathfrak{o})$.

The transitivity result we have to prove can now be rephrased as infinitesimal transitivity of the action \eqref{infi}.
This means to show that, given $\PP\in\Ga_\al(T^*M|_N)^\si$, 
for every $P'\in\Ga_0(T^*M|_N)$ there exists $(\mathbf{u},\nu)\in\X_N(M)\,\circledS\, \F_N(M,\mathfrak{o})$ such that 
\begin{equation}\label{prim}
P'=-\PP\o(\nabla\mathbf{u})|_N-\langle\si,(\dd\nu)|_N\rangle.
\end{equation}
For this we proceed as in \cite{GBVi2011}.
We fix a Riemannian metric on $M$ and denote by $ \nabla $, $\sharp$, and $\|\,\|$ the Levi-Civita covariant derivative, the sharp operator, and the norm associated to the Riemannian metric, respectively.
We also fix an inner product on the (finite dimensional) Lie algebra $\mathfrak{o}$
with induced norm $\|\,\|$ (both on $\mathfrak{o}$ and on $\mathfrak{o}^*$)
and induced isomorphism $\sharp:\mathfrak{o}^*\to\mathfrak{o}$.

Since $\PP$ and $\si$ cannot vanish simultaneously, the function
$\Vert \PP\Vert^2+\Vert\si\Vert^2$ is non-zero everywhere on $N$, so
we can consider the section 
\[
\la:=\frac{P'}{\|\PP\|^2+\|\si\|^2}\in\Ga_0(T^*M|_N).
\]
Lemma \ref{function} can be applied to the restriction of $\la$ to the normal bundle over the submanifold $N\subset M$ 
(viewed as a tubular neighborhood of $N$)
because the restriction of $\la$ to $TN$ vanishes.
We thus obtain from $\la$ a function $f$ on the normal bundle, linear on each fiber, with the property that the differential of $f$ along $N$ is $\la$,
\ie $(\dd f)|_{N}=\la$ as sections of $T^*M|_{N}$.
Using the tubular neighborhood of $N$ we build a smooth
function on $M$, identical to $f$ on a neighborhood of $N$, also denoted by $f$.
In particular $f$ vanishes on $N$ since $\la$ 
clearly vanishes on the zero section of $T^*M|_N$. 

From this function $f$ we define the pair $( \mathbf{u} , \nu ) \in \mathfrak{X}  _N (M) \times \mathcal{F} _N( M, \mathfrak{o} )$ by $\mathbf{u}:=-f\widetilde{\PP^\sharp}$ and $\nu:=-f\widetilde{\si^\sharp}$, where $\widetilde{\PP^\sharp}\in\X(M)$ is an arbitrary smooth extension of $\PP^\sharp:N\to TM$ and $\widetilde{\si^\sharp}\in\F(M,\mathfrak{o})$ is an arbitrary smooth extension of $\si^\sharp:N\to\mathfrak{o}$,
both obtained by a standard argument using a partition of unity.

The computation $-(\nabla \mathbf{u})|_N=(f\nabla \widetilde{\PP^\sharp})|_N+((\dd f) \widetilde{\PP^\sharp})|_N=(\dd f)|_N\PP^\sharp=\la \PP^\sharp$
implies that 
\[
-\PP\o (\nabla \mathbf{u})|_N=\PP\o(\la \PP^\sharp)
=\frac{\| \PP\|^2}{\|\PP\|^2+\|\si\|^2}P'.
\]
On the other hand $-(\dd\nu)|_N=(f\dd \widetilde{\si^\sharp})|_N+((\dd f) \widetilde{\si^\sharp})|_N
=(\dd f)|_N\si^\sharp=\la\si^\sharp$
implies that
\[
-\langle\si,(\dd\nu)|_N\rangle=\langle\si,\la\si^\sharp\rangle
=\frac{\| \si\|^2}{\|\PP\|^2+\|\si\|^2}P'.
\]
By adding them we get \eqref{prim},
which ensures the infinitesimal transitivity of the action \eqref{infi} of
$\X_N(M)\,\circledS\, \F_N(M,\mathfrak{o})$ on $\Ga_\al(T^*M|_N)^\si $.
\end{proof}

\medskip

The next Proposition is the analogue of Proposition \ref{transitivity1} but for the right action and the left momentum map.

\begin{proposition}
The group $\Aut(S\x\O)$ acts transitively on level sets of the momentum map
$\mathbf{J}_L$
given in \eqref{left_momap_trivial} restricted to $T^*Q_{KK}^\x$.
\end{proposition}

\begin{proof}
Suppose that $(\PP_\QQ,\ka_\ga)$ and $(\PP'_{\QQ'},\ka'_{\ga'})$
lie in the same level set of $\J_L$, \ie $\J_L(\PP_\QQ,\ka_\ga)=\J_L(\PP'_{\QQ'},\ka'_{\ga'})$. We thus have
\[
\int_S\PP_\QQ \cdot (X\o \QQ)\mu _S =\int_S\PP'_{\QQ'} \cdot (X\o \QQ')\mu _S \quad\text{and}\quad \int_S\ka_\ga\ga^{-1}(f\o \QQ)\mu _S 
=\int_S\ka'_{\ga'}(\ga')^{-1}(f\o \QQ')\mu _S ,
\] 
for all $X\in\X(M)$ and all $f\in\F(M,\ou)$.

These identities ensure that the embeddings $\QQ$ and $\QQ'$ have the same image: 
$ \mathbf{Q} (S)= \mathbf{Q} '(S)$.
In order to prove this, we fix $x_0 \in S$. If $ \mathbf{P}_ \mathbf{Q} (x_0 )\neq 0$, then from the first equality we obtain that $ \mathbf{Q} (x_0 ) \in \mathbf{Q} '(S)$. Indeed, if this is not the case, then we can find $X$ with compact support 
$K= \operatorname{supp}X $ such that $K\ni \mathbf{Q} (x_0 )$ and 
$K \cap \mathbf{Q} '(S)=\varnothing$.  Such an $X$ can also be chosen such that $\int_S\PP_\QQ \cdot (X\o \QQ)\mu _S \neq 0$, whereas we always have $\int_S\PP'_{\QQ'} \cdot (X\o \QQ')\mu _S=0$. This is in contradiction with the hypothesis. If $ \mathbf{P} _ \mathbf{Q} (x_0 )=0$, 
then, by the definition of $T^*Q_{KK}^\x$, $\ka_\ga(x_0 )\ne 0$. 
Then also $\ka_\ga(x_0 )\ga(x_0 )^{-1}\ne 0$ and we use the second identity with the same argument as above, to get $ \mathbf{Q} (x_0 ) \in \mathbf{Q} '(S)$. 
Doing this for all $x_0 \in S$ proves that $ \mathbf{Q} (S) \subset \mathbf{Q} '(S)$ and, similarly, that $ \mathbf{Q}' (S) \subset \mathbf{Q}(S)$.

Since $ \mathbf{Q} (S)= \mathbf{Q} '(S)$, there exists $\ps\in\Diff(S)$
such that $\QQ'=\QQ\o\ps$. 
Plugging this into the first identity above and using a change of variables, we get
\[
\int_S\PP_\QQ \cdot (X\o \QQ)\mu _S =\int_S (\PP'_{\QQ'} \circ \psi ^{-1} ) \cdot (X\o \QQ)\Jac_{\psi ^{-1}}\mu _S, 
\]
for all $X\in\X(M)$.
We know that $\Jac_{\ps^{-1}}=\Jac_{\ps}^{-1}\o\ps^{-1}$,
so $\PP'_{\QQ'}=(\PP_\QQ \circ \psi )\Jac_\ps$, the first component of \eqref{right_cotangent}. 

Let $b=(\ga^{-1}\o\ps)\ga'\in\F(S,\O)$, so $\ga'=(\ga\o \ps)b$, which we plug into the second identity above to get
\begin{align*}
\int_S\ka_\ga\ga^{-1}(f\o \QQ)\mu _S &=\int_S\ka'_{\ga'}b^{-1}(\ga^{-1}\o\ps)(f\o \QQ\o\ps)\mu _S \\
&=\int_S\left( (\ka'_{\ga'} b^{-1})\circ \psi ^{-1} \right) \ga^{-1}(f\o \QQ)\Jac_{\ps^{-1}}\mu _S .
\end{align*}
Since $f\o \QQ$ is arbitrary in $\F(S,\ou)$, 
we get $\ka_\ga=\left((\ka'_{\ga'} b^{-1})\o\ps^{-1}\right)\Jac_{\ps^{-1}}$.
This means that $\ka'_{\ga'}=(\ka_\ga\circ \psi )b\Jac_\ps$, 
the second component of \eqref{right_cotangent}.
\end{proof}

\medskip

From  the preceding propositions we obtain that the commuting actions 
of $\Aut(M\x\O)$ and $\Aut(S\x\O)$ are mutually completely orthogonal. We thus get the following result.

\begin{theorem}
The momentum maps \eqref{left_momap_trivial} and \eqref{momap_right_trivial} associated to the EPAut equations, 
form a dual pair:

\begin{picture}(150,100)(-70,0)%
\put(65,75){$T^*\Emb_\O(S\x\O, M\x\O)^\x$}

\put(90,50){$\mathbf{J}_L$}

\put(160,50){$\mathbf{J}_R$}

\put(52,15){$\mathfrak{aut}(M\times\mathcal{O})^*$
}

\put(150,15){$\mathfrak{aut}(S\times\mathcal{O})^*$.
}

\put(130,70){\vector(-1, -1){40}}

\put(135,70){\vector(1,-1){40}}

\end{picture}
\end{theorem}

Recall that the momentum map $ \mathbf{J} _L$ provides the formula for possible singular solutions of the EPAut equations on $M$  (\cite{GBTrVi2013}). Being equivariant, $ \mathbf{J} _L$ is a formally Poisson map relative to the canonical symplectic form on $T^*\Emb_\O(S\x\O, M\x\O)^\x$ and the Lie-Poisson structure on $\mathfrak{aut}(M\times\mathcal{O})^*$. This ensures that the parameterization of the singular solutions in terms of $\mathbf{P}_\mathbf{Q}$ and $\kappa_\ga$ are Clebsch
variables in the sense of \cite{MaWe83}. The dual pair property tells us that $\Aut(S\x\O)$ is the gauge group of that Clebsch representation.
Since $ \mathbf{J} _R$ is $\Aut(S\x\O)$-invariant, it follows that $ \mathbf{J} _R$ is a Noether conserved quantity for the canonical dynamics of these singular solutions.

Whereas the map  $ \mathbf{J} _L$ is a geometric object that is always well-defined, it is not always true that the corresponding EPAut equations admit these singular solutions.
This happens only for a certain class of Hamiltonians $h$ for which the expression $h \circ \mathbf{J} _L$ is well-defined. Such a class includes the modified Camassa-Holm equation (\cite{HoOnTr2009}) together with its higher dimensional and anisotropic versions studied in \cite{HoTr2008}.
It is interesting to mention that while it is well known that the strong solutions of these EPAut equations are described by geodesics of a right-invariant metric on $\Aut(M\x\O)$, the singular solutions also admit a geodesic interpretation. Indeed, this follows from a general result in \cite{GBRa2011b} (see Theorem 2.5), that the singular solutions given by $ \mathbf{J} _L$ are described by geodesics $ t \mapsto ( \mathbf{Q} (t), \gamma (t))$ on a $\Aut(M\x\O)$-orbit in $\Emb_\O(S\x\O, M\x\O)^\x$, relative to the normal metric associated to the right-invariant metric on $\Aut(M\x\O)$.

Being equivariant, the momentum map $ \mathbf{J} _R$ also yields Clebsch variables for the EPAut equation on $S$. The dual pair property again ensures that $\Aut(M\x\O)$ is the gauge group the Clebsch representation.


\section{The incompressible EPAut equation and momentum maps}\label{secti}

In this section we recall the expression of the Euler-Poincar\'e equations on the group of volume preserving automorphisms of a trivial principal bundle (the EPAut$_{\rm\,vol}$ equations) and review from \cite{GBTrVi2013} some facts about the associated pair of momentum maps. We will then focus on a particular case relevant for the Yang-Mills Vlasov equations, arising when the total space of one of the principal bundles is a cotangent bundle (the so called Yang-Mills phase space).

\subsection{The group of volume preserving automorphisms}

Let $\pi:P\rightarrow M$ be a principal $\O$-bundle and suppose that $M$ is orientable, endowed with a Riemannian metric $g$. Let $\mu_M$ be the volume form induced by $g$.

The group $\Aut_{\vol}(P)$ consists, by definition, of the automorphisms of the principal bundle $P$ which descend to volume preserving diffeomorphisms of the base manifold $M$ with respect to the volume form $\mu_M$. Its Lie algebra, denoted by $\mathfrak{aut}_{\vol}(P)$ consists of equivariant vector fields such that their projection to $M$ is divergence free.

\begin{remark}[Kaluza-Klein metric and induced volume]\label{rem_KK}
{\rm
Given a principal connection $\mathcal{A}\in \Omega ^1 (P, \mathfrak{o} )$ on $P$
and an inner product $\tau $ on $\mathfrak{o}$, one defines the Kaluza-Klein Riemannian metric on $P$ as
\[
\ka(U_p,V_p)=g(T\pi(U_p),T\pi(V_p))+\tau (\A(U_p),\A(V_p)),\quad U_p, V_p\in T_pP.
\]
The induced volume form $\mu_P$ on $P$ is given by
\[
\mu_P=\pi^*\mu_M\wedge\mathcal{A}^*\rm{det} _\tau ,
\]
where $\mathcal{A}^*\det_\tau $ denotes the pullback by the connection $\mathcal{A}:TP\to\mathfrak{o}$ of the canonical determinant form induced by $\tau $ on $\mathfrak{o}$.
Supposing that $\tau $ is $\operatorname{Ad}$-invariant, the Kaluza-Klein metric $\ka$ and the volume form $\mu_P$ induced by $\ka$ are $\mathcal{O}$-invariant.
Given $\varphi\in\mathcal{A}ut(P)$, we have the equivalence
\[
\varphi\in\mathcal{A}ut_{\vol}(P)\Leftrightarrow \varphi^*\mu_P=\mu_P,
\]
see Lemma 4.1 in \cite{GBTrVi2013}.
As a consequence, the Lie algebra $\mathfrak{aut}_{\vol}(P)$ coincides with the Lie algebra of equivariant divergence free vector fields with respect to $\mu_P$.}
\end{remark}

When the principal bundle is trivial, i.e. $\pi :P\simeq M\times \mathcal{O}\rightarrow M$, the group $\Aut_{\rm vol}(P)$  is isomorphic to the semidirect product group
\[
\mathcal{A}ut_{\rm vol}(P)\simeq\operatorname{Diff}_{\rm vol}(M)\,\circledS\,\mathcal{F}(M,\mathcal{O}),
\]
where $\operatorname{Diff}_{\rm vol}(M)$ acts on $\mathcal{F}(M,\mathcal{O})$
by composition on the right.  Its Lie algebra is the semidirect product Lie algebra
$\mathfrak{aut}_{\rm vol}(P)\simeq\mathfrak{X}_{\rm vol}(M)\,\circledS\,\mathcal{F}(M,\mathfrak{o})$.
Using the $L^2$ pairing associated to the volume form $ \mu _M $, we identify the dual as
\begin{align*}
\mathfrak{aut}_{\vol}(P)^*\simeq\X_{\vol}(M)^*\x\F(M,\ou)^*=\Om^1(M)/\mathbf{d} \F(M)\x\F(M,\ou^*).
\end{align*}
The Euler-Poincar\'e equations
on the automorphism group $\mathcal{A}ut_{\rm vol}(P)$ (the EPAut$_{\rm\,vol}$ equations) take the following form when $P$ is a trivial principal bundle.

\begin{proposition}[The EPAut$_{\rm vol}$ equations on a trivial principal bundle] 
{\rm\cite{GBTrVi2013}}\label{EPaut_vol_trivial}
Consider a reduced Lagrangian $\ell:\mathfrak{X}_{\rm vol}(M)\,\circledS\,\mathcal{F}(M,\mathfrak{o}) \to\RR$. Then the associated EPAut$_{\vol}$ equations are 
\begin{equation}\label{EPAut_vol_trivial}
\left\{\begin{array}{l}
\vspace{0.2cm}\displaystyle
\frac{\partial}{\partial t} \frac{\delta \ell}{\delta\mathbf{u}} +\pounds_{\mathbf{u}}\frac{\delta \ell}{\delta\mathbf{u}}+\frac{\delta \ell}{\delta{\nu}}\!\cdot\!\mathbf{d}{\nu}=-\dd p\\
\displaystyle\frac{\partial}{\partial t}\frac{\delta \ell}{\delta{\nu}}+\operatorname{ad}^*_{{\nu}}\frac{\delta \ell}{\delta{\nu}}+\dd\left(\frac{\delta \ell}{\delta{\nu}}\right)\cdot\mathbf{u}=0,
\end{array}\right.
\end{equation}
where $(\mathbf{u},{\nu})\in\mathfrak{X}_{\rm vol}(M)\,\circledS\,\mathcal{F}(M,\mathfrak{o})$ and the operator $\pounds_\mathbf{u}$ denotes the Lie derivative acting on one-forms. The first equation is written in $ \Omega  ^1 (M)$ and $p\in\F(M)$ denotes the pressure, determined from the incompressibility condition $ \operatorname{div} \mathbf{u} =0$.
\end{proposition}

We refer to \cite{GBTrVi2013} for the description of the  EPAut$_{\rm vol}$ equations on an arbitrary principal bundle.


\subsection{Review of the ideal fluid case} 

The pair of momentum maps associated to the Euler equations 
discovered in \cite{MaWe83} justifies geometrically the existence of Clebsch canonical variables for ideal fluid motion and explains the Hamiltonian structure of point vortex solutions in terms of the Hamiltonian structure of the Euler equations. As claimed in \cite{MaWe83}, and rigorously shown in \cite{GBVi2011}, this pair of momentum maps forms a {dual pair}.

Given a compact volume manifold $(S, \mu )$ and a symplectic manifold $(M, \omega )$, the pair of momentum maps arise from the commuting symplectic actions of the groups $ \operatorname{Diff}_{\symp}(M)$ and $ \operatorname{Diff}_{\vol}(S)$ on the Fr\'echet manifold $ \Emb(S,M)$ endowed with the symplectic form
\[
\bar\omega(f)(u_f,v_f):=\int_{S}\omega(f(x))(u_f(x),v_f(x))\mu_{S}.
\]
In order to get Hamiltonian actions it is needed to replace these groups by the subgroups $\operatorname{Diff}_{\ham}(M)$ and $\operatorname{Diff}_{\ex}(S)$ of Hamiltonian and exact volume preserving diffeomorphisms, respectively. Furthermore, in order to have equivariance,
required from the dual pair properties, it is needed to consider central extensions of these groups, given by the group of quantomorphisms (central extension of 
$\Diff_{ham}(M)$) and the Ismagilov
central extension of $ \operatorname{Diff}_{\ex}(S )$, respectively (\cite{GBVi2011}). 

In the particular case when  the symplectic form is exact, $ \omega=-\dd\theta $, one can stay with the whole group $ \operatorname{Diff}_{\vol}(S)$ (instead of  $ \operatorname{Diff}_{\ex}(S)$) and the central extension is not needed. Another simplification arises in this case, since the quantomorphism group can be written as a topologically trivial extension of $ \operatorname{Diff}_{\ham}(M)$, with the help of a group $2$-cocycle $B$ (\cite{IsLoMi2006}), 
\[
B(\ph_1,\ph_2):=\int_{m_0}^{\ph_2(m_0)}\left(\theta-\ph_1^*\theta\right),\quad \ph_1,\ph_2\in\Diff_{\ham}(M),
\]
extension denoted by $\operatorname{Diff}_{\ham}(M) \times _B \mathbb{R}$.


\paragraph{The right action momentum map.}
The natural action of the group $\Diff_{\vol}(S)$ on $\Emb(S,M)$ is Hamiltonian with equivariant momentum map $\J_R(f)=[f^*\th]$, where $\omega=-\dd\theta$.
Here the dual of $\X_{\vol}(S)$ is identified with the quotient space $\Om^1(S)/\mathbf{d} \Om^0(S)$.
If in addition $H^1(S)=0$, then this dual can be identified with the space $\dd\Om^1(S)$ of vorticities,
and the right momentum map  becomes $\J_R(f)=-f^*\omega$, \cite{MaWe83}.

\begin{lemma}\label{lem3}
$\Diff_{\ham}(M)$ acts transitively on connected components of level sets of the right leg momentum map $\J_R:\Emb(S,M)\to \Om^1(S)/\dd\Om^0(S)$, $\J_R(f)=[f^*\th]$.
\end{lemma}
\begin{proof}
The transitivity of the action of $\Diff_{\ham}(M)$ 
on connected components of level sets of $\J_R$ follows from the
transitivity of the action of $\X_{\ham}(M)$ on level sets of $\J_R$,
since the constructions can be performed smoothly depending on a parameter.
We start again with an arbitrary vector field $v_f$ on $M$ along $f$ such that
$T_f\J_R\cdot v_f=[f^*\pounds _{v_f}\th]=0$.
There exists a function $\bar h\in C^\oo(S)$ whose differential is $\mathbf{d} \bar h=-f^*\mathbf{i} _{v_f}\mathbf{d} \th$.
We extend it to a function $h_1\in C^\oo(M)$ such that $\bar h=h_1\o f$.
Now the 1-form $\be$ on $M$ along $S$ defined by
\[
\be=\mathbf{d} h_1\o f+\mathbf{i} _{v_f}(\mathbf{d} \th\o f)\in\Ga(f^*T^*M)
\] 
vanishes on vectors tangent to $f(S)\subset M$. 
By Lemma 4.2 from \cite{GBVi2011} we find $h_2\in C^\oo(M)$ such that $\be=\mathbf{d} h_2\o f$.
It follows that $\mathbf{d} (h_1-h_2)\o f=-\mathbf{i} _{v_f}d\th=\mathbf{i} _{v_f}\om$,
so $v_f=X_{h_1-h_2}\o f\in\X_{\ham}(M)_{\Emb}(f)$.
\end{proof}

\paragraph{The left action momentum map.}
The Lie algebra of the quantomorphism group is the central extension $C^\oo(M)$ of the Lie algebra of Hamiltonian vector fields
\begin{gather*}
h\in C^\oo(M)\;\longmapsto\; X_h\in\X_{\ham}(M),\quad \mathbf{i} _{X_h}\om=\mathbf{d} h.
\end{gather*}
Its dual can be identified with the space of compactly supported densities $\Den_c(M)$, 
so the infinitesimally equivariant left leg momentum map is $\J_L(f)=f_*\mu$,  \cite{MaWe83}.

\begin{lemma}\label{lem1}{\rm\cite{GBVi2011}}
$\Diff_{\vol}(S)$ acts transitively on level sets of the left leg momentum map $\J_L:\Emb(S,M)\to
\Den_c(M)$.
\end{lemma}

\begin{proof}
Let $f_1,f_2\in\Emb(S,M)$ such that $\J_L(f_1)=\J_L(f_2)$.
Then we have 
\begin{equation}\label{inte}
\int_S(h\o f_1)\mu=\int_S(h\o f_2)\mu, \quad h\in C^\oo(M).
\end{equation}
A first consequence is that the two embeddings have the same image in $M$,
so there exists $\ps\in\Diff(S)$ such that $f_2=f_2\o\ps$.
We rewrite the identity \eqref{inte} as $\int_S(h\o f_2)\ps^*\mu=\int_S(h\o f_2)\mu$
for all $h\o f_2\in C^\oo(S)$, and we deduce $\ps^*\mu=\mu$.
We found $\ps\in\Diff_{\vol}(S)$ such that $f_2=f_1\o\ps$.
\end{proof}

\medskip

The group of volume preserving diffeomorphisms $\Diff_{\vol}(S)$
and the quantomorphism group $\Diff_{\ham}(M)\x_B\RR$ have mutually completely orthogonal actions on the manifold of embeddings  $\Emb(S,M)$, so 

\begin{picture}(150,100)(-70,0)%
\put(105,75){$\Emb(S,M)$}

\put(90,50){${J}_L$}

\put(160,50){${J}_R$}

\put(2,15){$(M)=C^\oo(M)^*$}

\put(160,15){$\X_{\vol}(S)^*=\Om^1(S)/\mathbf{d} \Om^0(M)$}

\put(130,70){\vector(-1, -1){40}}

\put(135,70){\vector(1,-1){40}}

\end{picture}

\noindent is a dual pair. 


\subsection{A pair of momentum maps for the EPAut$_{\rm vol}$ equations}

The above setting for the ideal fluid equations 
has been developed in \cite{GBTrVi2013} for the EPAut$_{\rm vol}$ equations
as follows.

Let $\pi_S:P_S\rightarrow S$ be a principal $\O$-bundle and consider another principal $\O$-bundle $\pi_M :P_M\rightarrow M$ such that $P_M$ carries an exact symplectic form $\omega=-\mathbf{d}\theta$, where $\theta\in \Omega ^1 (P_M)$ is $\mathcal{O}$-invariant. Assume that both $S$ and $ \mathcal{O} $ are compact,
hence $P_S$ is compact too.
As above, we endow $P_S$ with the $\mathcal{O}$-invariant volume form
$\mu_{P_S}=\pi^*\mu_S\wedge\mathcal{A}^*\operatorname{det}_\tau $, where $\tau $ is an $\operatorname{Ad}$-invariant inner product on
$\mathfrak{o}$ and $\A$ a principal connection on $P_S$.

The space  of $ \mathcal{O} $-equivariant embeddings from $P_S$ into $P_M$,
denoted by $Q_{KK}=\Emb_{\O}(P_S,P)$, is a Fr\'echet manifold
because $P_S$ is compact (see Remark \ref{smooth}).
We endow the manifold $Q_{KK}$ with the symplectic form $\bar\omega$ given by
\[
\bar\omega(f)(u_f,v_f):=\int_{P_S}\omega(f(p))(u_f(p),v_f(p))\mu_{P_S}.
\]
The function under the integral is $\mathcal{O}$-invariant, so the right hand side can be written as an integral over $S$. 
The local triviality of the bundle $P_S\to S$ ensures the non-degeneracy of $\bar\om$.

We describe below the momentum maps associated to the EPAut$_{\rm \, vol}$ equations. In this context, the manifold $\Emb(S,M)$ and the groups $ \operatorname{Diff}_{\vol}(S, \mu )$  and $ \operatorname{Diff}_{\ham}(M) \times _B \mathbb{R}  $ of the ideal fluid case will be replaced by the manifold  $\Emb_{\O}(P_S,P_M )$, the group $\mathcal{A}ut_{\vol}(P_S)$ of volume preserving automorphisms, and the group $\mathcal{VC}hrom(P_M)$ of Vlasov chromomorphisms, respectively.

\color{black}

\subsubsection{Left action momentum map}

Let us denote by $\Aut_{\ham}(P_M):={\mathcal{A}ut(P_M )\cap \operatorname{Diff}_{\ham}(P_M )}$
the group of Hamiltonian automorphisms of $P_M$ whose Lie algebra $ \mathfrak{aut}_{\ham}(P_M)$ consists of $ \mathcal{O} $-equivariant Hamiltonian vector fields on $P_M$. This group acts symplectically by composition on the left on $\operatorname{Emb}_\mathcal{O}(P_S,P_M)$ and admits a momentum which is not equivariant and hence not Poisson. As in \cite{GBVi2011}, in order to obtain an equivariant momentum map, we have to consider the central extension of
$\mathcal{A}ut_{\ham}(P_M)$ by the cocycle  
\begin{equation}\label{cocycle}
B(\ph_1,\ph_2):=\int_{p_0}^{\ph_2(p_0)}\left(\theta-\ph_1^*\theta\right),\quad \ph_1,\ph_2\in
\Aut_{\ham}(P_M),
\end{equation}
\color{black}where the integral is taken along a smooth curve connecting the point $p_0$ with the point $\ph_2 (p_0)$.
The cohomology class of $B$ is independent of the choice of the point $p_0$ and the one-form
$ \theta$ such that $-\mathbf{d}\theta= \omega $, see Theorem 3.1 in \cite{IsLoMi2006}. As shown in \cite{GBTrVi2013}, in order to obtain an equivariant momentum map, one needs to consider the subgroup $\overline{\mathcal{A}ut}_{\ham}(P_M )$ of $\mathcal{A}ut_{\ham}(P_M )$, whose Lie algebra consists of Hamiltonian vector fields associated to $\mathcal{O}$-invariant Hamiltonian functions on $(P_M,\omega)$.

This group is referred to as the group of  \textit{special Hamiltonian automorphisms}.
Its central extension $\mathcal{VC}hrom(P_M):=\overline{\mathcal{A}ut}_{\ham}(P_M)\times_B\mathbb{R}$ is called the \textit{Vlasov chromomorphism group} since it is the configuration Lie group for Yang-Mills Vlasov plasmas in chromohydrodynamics. The Lie algebra of $\mathcal{VC}hrom(P_M )$ is isomorphic to the space of functions on $M$ whose Lie bracket is given by the reduced Poisson bracket on $M=P _M / \mathcal{O} $ obtained by reduction of the symplectic Poisson bracket on $(P_M , \omega )$. We refer to \cite{GBTrVi2013} for more details regarding the definition of these groups.

The group $\mathcal{VC}hrom(P_M )$ acts symplectically on the left on the symplectic manifold $(\operatorname{Emb}_\mathcal{O}(P_S,P),\bar\omega)$ and admits the momentum map
\begin{equation}\label{left_momap_autvol}
\mathbf{J}_L:\Emb_{\O}(P_S,P_M )\rightarrow\mathcal{F}(M)^*=\mathcal{F}_\mathcal{O}(P_M )^*,\quad  
\left\langle\mathbf{J}_L(f), h\right\rangle=\int_{P_S}\tilde h(f(p))\mu_{P_S},
\end{equation}
where $h \in \mathcal{F} (M)$ and $\tilde h=h\o\pi_M\in \mathcal{F} _ \mathcal{O} (P_M)$. Since the function $p\mapsto\tilde h(f(p))$ on $P_S$ is $\mathcal{O}$-invariant, it defines a function on $S$, and we have in fact an integral over $S$.
By abuse of notation, we can write
\[
\mathbf{J}_L(f)=\int_S\delta (n-f(p))\mu_S\in\mathcal{F}(M)^*.
\]

\begin{remark}[Special Hamiltonian automorphisms]{\rm It is interesting to recall here that the group $ \operatorname{Diff}_{\ham}(M, \omega )$ is of central importance in Hamiltonian mechanics, since it contains the flows of Hamiltonian systems on the symplectic manifold $(M, \omega )$. In the same way, the group  $\overline{\mathcal{A}ut}_{\ham}(P_M, \omega )$ of special Hamiltonian automorphisms, where $P_M \rightarrow M$ is a $ \mathcal{O} $-principal bundle and the symplectic form $ \omega $ is $ \mathcal{O} $-invariant, is the corresponding group in the case of Hamiltonian systems with symmetries. 
It contains the flows of Hamiltonian systems with $ \mathcal{O} $-symmetries.

An important example in this context are Wong's equations for a nonabelian charged particle in a fixed Yang-Mills field. They arise as a Hamiltonian system
on $(P_M, \omega )=(T^*P, \Omega _{\rm can})$, where $P \rightarrow Q$ is a $ \mathcal{O} $-principal bundle over the physical space $Q$ of the particle, see \cite{Mo84}.
}
\end{remark}

\subsubsection{Right action momentum map}

The group $\mathcal{A}ut_{\vol}(P_S)$ acts symplectically on the right on the symplectic manifold $(\operatorname{Emb}_\mathcal{O}(P_S,P),\bar\omega)$ and admits the momentum map
\begin{equation}\label{right_momap_autvol}
\mathbf{J}_R:\operatorname{Emb}_\mathcal{O}(P_S,P_M )\rightarrow\mathfrak{aut}_{\vol}(P_S)^*=\Omega^1_\mathcal{O}(P_S)/\mathbf{d}\F_{\mathcal{O}}(P_S),\quad \mathbf{J}_R(f)=[f^*\theta].
\end{equation}
{The identification $\mathfrak{aut}_{\vol}(P_S)^*=\Omega^1_\mathcal{O}(P_S)/\mathbf{d}\F_{\mathcal{O}}(P_S)$ is made by using the duality pairing induced by the duality pairing
\begin{equation}\label{L2pairing} 
\left\langle \alpha , X \right\rangle :=\int_{P_S} \alpha ( X) \mu _{P_S}
\end{equation} 
between $\mathfrak{aut}(P_S)$ and $ \mathfrak{aut}(P_S) ^\ast= \Omega ^1 (P_S)$. Note that in \eqref{L2pairing} the function $ \alpha ( X)$ is $ \mathcal{O} $-invariant since both $X$ and $ \alpha $ are equivariant, so it induces a function on $S$ so that the duality pairing can be rewritten as an integral over $S$ with respect to $ \mu _S $.

\subsubsection{The pair of momentum maps}

In summary, we have the following pair of momentum maps associated to the EPAut$_{\rm\, vol}$ equation

\begin{picture}(150,100)(-70,0)%
\put(102,75){$\Emb_{\O}(P_S,P_M)$}

\put(90,50){$\mathbf{J}_L$}

\put(160,50){$\mathbf{J}_R$}

\put(60,15){$\mathcal{F}(M)^*$
}

\put(160,15){$\mathfrak{aut}_{\vol}(P_S)^*$.
}

\put(130,70){\vector(-1, -1){40}}

\put(135,70){\vector(1,-1){40}}

\end{picture}

Being equivariant, the momentum maps $ \mathbf{J} _L$ and $ \mathbf{J} _R$ are Poisson maps and hence yield Clebsch variables for the Lie-Poisson systems on $\mathcal{F}(M)^\ast $ and the EPAut$_{\rm \, vol}$ equations on $P_S$, respectively. Note that the Lie-Poisson system on $\mathcal{F}(M)^\ast$ is a Vlasov system whose Poisson bracket is not symplectic but is the reduced Poisson bracket on $M= P_M/\mathcal{O} $.
The momentum map $ \mathbf{J} _L $ provides the expression of singular solutions for this Lie-Poisson system.

When $P_M=T^*\bar P$, where $\bar P$ is itself a $ \mathcal{O}$-principal bundle, this system
is related to Yang-Mills Vlasov equation and in this case $ \mathbf{J} _L$ can be identified with the single particle
solution, which is of central importance for the theory of Yang-Mills
charged fluids, \cite{GiHoKu1983}. This particular setting will be considered in
the following section.

Note also that, since both momentum maps are invariant under the action of the group associated to their partner momentum map, they also provide Noether conserved quantities for these Clebsch variables. The dual pair property that will be shown below will allow us to identify the gauge groups of these Clebsch representations.

\subsection{Yang-Mills phase space}\label{yamil}

We now consider the special case when the total space of the principal
bundle $\pi _M :P_M \rightarrow M$ is the cotangent bundle of another
principal bundle $\bar \pi : \bar P\rightarrow \bar M$. We endow $P_M =T^*\bar P$ with the canonical symplectic form $\Omega_{\bar
P}=-\mathbf{d}\Theta_{\bar P}$ and we let $\mathcal{O}$ act on $T^*\bar P$ by cotangent lift. Thus we have $P_M=T^*\bar P \rightarrow M=T^*\bar P/\mathcal{O}$. 
This particular choice is motivated by the example of the Yang-Mills Vlasov equation, as mentioned in \cite{GBTrVi2013}, in which case $\bar \pi : \bar P \rightarrow \bar M$ is the principal bundle of the Yang-Mills theory involved.
If moreover $\bar \pi :\bar P \simeq\bar M\times\mathcal{O}\rightarrow M$ is a trivial $\O$-bundle, then we have $P_M =T^*\bar M\times T^*\mathcal{O}$ and $M=T^*\bar M\times\mathfrak{o}^*$, so that the left momentum map \eqref{left_momap_autvol} takes value in the space $ \mathcal{F} (T^*\bar M \times \mathfrak{o} ^\ast ) ^\ast $ of Yang-Mills Vlasov distributions. Note that in this case $\pi _M : P_M\rightarrow M$ is also a trivial principal $\O$-bundle, since we have the equivariant diffeomorphism
\begin{equation}\label{equiv_diffeo}
\rho: T^*\bar M\times T^*\mathcal{O} \rightarrow  \left(T^*\bar M\times \mathfrak{o}^*\right)\times\mathcal{O},\quad
\rho(\alpha_q,\alpha_g)=\left((\alpha_q,\alpha_gg^{-1}),g\right).
\end{equation}
If  the bundle $\pi _S :P_S \simeq S\x\O \rightarrow S$ is also trivial, then we have the identification
\begin{equation}\label{identification} 
Q_{KK}=\Emb_{\O}(P_S,T^*\bar P)\simeq\Emb(S,T^*\bar M\times\mathfrak{o}^*)\x\F(S,\mathcal{O})
\end{equation} 
(see Lemma 3.8 in \cite{GBTrVi2013}).
More precisely, to the equivariant embedding $f:S\times\mathcal{O}\rightarrow T^*\bar M\times T^*\mathcal{O}$, $f(x,g)=\left(\PP_\QQ(x),\kappa_\ga(x)g\right)$, where $\PP_\QQ:S\rightarrow T^*\bar M$ and $\kappa_\ga:S\rightarrow T^*\mathcal{O}$,
we associate the pair $((\PP_\QQ,\si),\ga)\in \Emb(S,T^*\bar M\times\mathfrak{o}^*)\x\F(S,\mathcal{O})$, where $ \sigma := \kappa _\gamma \gamma ^{-1} : S \rightarrow \mathfrak{o} ^\ast $. 
Note also that, since the bundle $P_S$ is trivial, we have $ \mu _{P_S}= \mu _S \wedge \operatorname{det}_ \tau $. 
Choosing the Ad-invariant inner product $ \tau  $ such that $ \operatorname{Vol} ( \mathcal{O} )=1$, we have  $ \mu _{P_S}= \mu _S \wedge \mu _ \mathcal{O} $, where $ \mu _ \mathcal{O} $ is the Haar measure on the compact group $\mathcal{O}$.

The pair of momentum maps \eqref{right_momap_autvol} and \eqref{left_momap_autvol} becomes

\begin{picture}(150,100)(-70,0)%
\put(80,75){$\Emb_{\O}(S\times \mathcal{O} ,T^*\bar M\times T^*\mathcal{O})$}

\put(90,50){$\mathbf{J}_L$}

\put(160,50){$\mathbf{J}_R$}

\put(50,15){$\mathcal{F}(T^*\bar M\times\mathfrak{o}^*)^*$
}

\put(160,15){$\mathfrak{aut}_{\vol}(S\x\O)^*$.
}

\put(130,70){\vector(-1, -1){40}}

\put(135,70){\vector(1,-1){40}}

\end{picture}

Notice that the Lie bracket on 
$\mathcal{F}(T^*\bar M\times\mathfrak{o}^*)$ is the reduced Poisson bracket given here by
\[
\{f,g\}_{M}=\{f,g\}_{T^*\bar M}+\{f,g\}_+,\quad f, g \in \mathcal{F}(T^*\bar M\times\mathfrak{o}^*),
\]
where the first term denotes the canonical Poisson bracket on
$T^*\bar M$ and the second term is the Lie-Poisson bracket on
$\mathfrak{o}^*$,
\[
\{f,g\}_{+}(\si)=\left\langle \si,\left[ \frac{\de f}{\de \si},\frac{\de g}{\de \si}\right] \right\rangle ,\quad f,g\in C^\oo(\mathfrak{o}^*).
\]
It is obtained by Poisson reduction of the canonical Poisson bracket on $T^*(\bar M \times \mathcal{O}) $.

The expression of these momentum maps were obtained in \cite{GBTrVi2013}. For later use, we provide below some details concerning their derivation.


\subsubsection{Left action momentum map}

By specifying formula \eqref{left_momap_autvol} to our case, we can write the left momentum map as
\[
\mathbf{J}_L:\Emb_{\O}(S\times\mathcal{O}, T^*\bar M\times T^*\mathcal{O})\rightarrow\mathcal{F}_\mathcal{O}(T^*\bar M\times T^*\mathcal{O})^*=\mathcal{F}(T^*\bar M\times\mathfrak{o}^*)^*,
\]
where, for $h \in \mathcal{F}(T^*\bar M\times\mathfrak{o}^*)$, $\tilde h=h\o\pi_M\in \mathcal{F}(T^*\bar M\times T^*\mathcal{O} )$ and $f(x,g)=((\PP_\QQ(x), \kappa _\gamma (x)g)$,
\begin{align*}
\left\langle\mathbf{J}_L(f),{h}\right\rangle
=\int_{S\times\mathcal{O}}\tilde h(\PP_\QQ(x),\kappa_\ga(x)g)\mu_{P_S}
=\int_S h\left(\PP_\QQ(x),\kappa_\ga(x)\ga(x)^{-1}\right)\mu_S.
\end{align*}
Formally, using the identification $f=((\PP_\QQ,\si),\ga)$ given in \eqref{identification}, this can be written as
\begin{equation}\label{unu}
\J_L(\PP_\QQ,\si,\ga)=(\PP_\QQ,\si)_*\mu_S.
\end{equation}



\subsubsection{Right action momentum map}

By specifying formula \eqref{right_momap_autvol} to our case, we can write the momentum map of the right action as
\begin{gather*}
\mathbf{J}_R:\Emb_{\O}(S\times\mathcal{O}, T^*\bar M\times T^*\mathcal{O})\rightarrow
\mathfrak{aut}_{\vol}(S\x\O)^*=\Om^1_\O(S\x\O)/\dd\F(S)\\
\J_R(f)=[f^*\Th_{\bar P}]=[f^*(\Th_{\bar M}+\Th_\O)],
\end{gather*}
where $\Th_{\bar P}$, $\Theta_{\bar M}$, $\Theta_\mathcal{O}$ are the canonical one-forms on 
$T^*\bar P$, $T^*{\bar M}$, $T^*\mathcal{O}$, respectively.

The identification of $\X_{\vol}(S)\,\circledS\, \F(S,{\mathfrak{o} })$ with the Lie algebra $\mathfrak{aut}_{\vol}(S\x\O)$
of  invariant divergence free vector fields on
$S\x\O$, namely $(u,\xi)(x,g)={(u(x),\xi(x)g)}$,
provides an identification of their duals: 
$(\X_{\vol}(S)\,\circledS\, \F(S,\O))^*=\left(\Omega^1(S)/\mathbf{d}\mathcal{F}(S)\right)\times\mathcal{F}(S,\mathfrak{o}^*)$ with $\mathfrak{aut}_{\vol}(S\x\O)^*=\Om^1_\O(S\x\O)/\dd\F(S)$, {via the map $([\al], \nu ) \mapsto [ (\alpha , \nu )]$}, where
$(\al,\nu)(v_x,\xi_g)=\al(v_x)+\nu(x)(\xi_gg^{-1})$.

We now show that in terms of the identification $f=((\PP_\QQ,\si),\ga)$ in \eqref{identification}, the right momentum map has the expression
\begin{equation}\label{doi}
\J_R((\PP_\QQ,\si),\ga)
=\left( \left[\PP_\QQ^*\Theta_{\bar M}+\langle\si, \de^r\ga\rangle\right],\Ad^*_\ga\si\right)\in \left(\Omega^1(S)/\mathbf{d}\mathcal{F}(S)\right)\times\mathcal{F}(S,\mathfrak{o}^*).
\end{equation}

Knowing that $\ka_\ga=\si\ga$, the first component $\al\in\Om^1(S)$ of 
$f^*\Th_{\bar P}$ 
reads
\begin{align*}
\al(v_x)=(f^*(\Th_{\bar M}+\Th_\O))(v_x,0_e)
=\PP_\QQ^*\Theta_{\bar M}(v_x)+\kappa_\ga^*\Theta_\mathcal{O}(v_x),
\end{align*}
for all $ v _x \in T_xS$. Using the definition of the canonical one-form $\Th_\O$, we have
\begin{align*}
(\ka_\ga^*\Th_\O)(v_x)
=\left\langle \ka_\ga(x),T_x(\pi\o\ka_\ga)\cdot v_x\right\rangle =\left\langle \si(x)\ga(x),T_x\ga \cdot v_x\right\rangle 
=\left\langle \si,\de^r\ga\right\rangle (v_x),
\end{align*}
{where $ \pi :T^* \mathcal{O} \rightarrow \mathcal{O} $}.

Given $\be\in T^*\O$, we denote by $\ell_\be:\O\to T^*\O$ the orbit map defined by $\ell_\be(g)=\be g$.
For the second component $\nu\in\F(S,\ou^*) $ of $f^*\Th_{\bar P}$ we compute
for all $\xi\in\mathfrak{o}$:
\begin{align*}
\langle\nu(x),\xi\rangle
&=(f^*\Th_{\bar P})(0_x,\xi)
=\Th_\O(T\ell_{\ka_\ga(x)}(\xi))
=\left\langle \ka_\ga(x), (T\pi\o T\ell_{\ka_\ga(x)})(\xi)\right\rangle \\
&=\left\langle \ka_\ga(x),\ga(x)\xi\right\rangle = \langle\ga(x)^{-1}\ka_\ga(x),\xi\rangle,
\end{align*}
where we used the identity $(T\pi\o T\ell_\be)(\xi)=\pi(\be)\xi$, valid for all $\be\in T^*\O$. 
Hence we obtain $\nu=\ga^{-1}\ka_\ga=\ga^{-1}\si\ga=\Ad^*_\ga\si$. 
By combining the above formulas we get the desired expression \eqref{doi}. 

Note that when $H ^1 (S)=0$, then the dual space $\left(\Omega^1(S)/\mathbf{d}\mathcal{F}(S)\right)\times\mathcal{F}(S,\mathfrak{o}^*)$ is isomorphic to the space $\mathbf{d} \Omega^1(S)\times\mathcal{F}(S,\mathfrak{o}^*)$ and we can write the first component of the momentum map as a vorticity as follows
\[
\mathbf{J} _R((\PP_\QQ,\si),\ga)= \left( -\PP_\QQ^*\Omega _{\bar M}-( \sigma \gamma ) ^\ast\Omega _\mathcal{O}, \Ad^*_\ga\sigma \right),
\]
where $\Omega _{\bar M}= - \mathbf{d} \Theta  _{\bar M}$ and $\Omega _\mathcal{O}= - \mathbf{d} \Theta  _\mathcal{O}$ are the canonical symplectic forms on $T^*\bar M$ and $T^* \mathcal{O} $, respectively.


\section{The dual pair property of the EPAut$_{\rm vol}$ momentum maps}\label{sec_5} 

In this section we will focus on the particular case of the Yang-Mills phase space described in \S\ref{yamil}. We assume that the principal bundles are trivial, i.e. $P_M=T^*\bar P$, $\bar P\simeq \bar M \times \mathcal{O} $, $P_S\simeq S \times \mathcal{O} $, so the commuting actions of the groups $\Aut_{\vol}(S \x \O)$ and $\overline\Aut_{\ham}(T^*\bar M\x T^*\O)$ on the symplectic manifold $Q_{KK}$  become
\begin{align*}
(\et,\ga)\cdot (\ps,b)&=(\et\o\ps,(\ga\o\ps)b),\quad\ps\in\Diff(S),\;\;b\in\F(S,\O)\\
(\ph,a)\cdot (\et,\ga)&=(\ph\o\et,(a\o \et)\ga),\quad\ph\in\Diff(T^*\bar M\x \ou^*),\;\;a\in\F(T^*\bar M\x \ou^*,\O),
\end{align*}
for $(\et,\ga)\in Q_{KK}\simeq\Emb(S,T^*\bar M\times\mathfrak{o}^*)\x\F(S,\mathcal{O})$ (see \eqref{identification}). The associated momentum maps have been described in \eqref{unu} and \eqref{doi}.

In the next two propositions, we shall show the transitivity results needed to obtain the dual pair property of these momentum maps.
We will use the following Lemma, a direct generalization of the corresponding formula on Lie groups, i.e. when the manifold $\bar M$ is absent, see e.g.
Proposition 13.4.3. in \cite{MaRa99}.

\begin{lemma}\label{mara}
Let $\tilde h\in\F(T^*\bar M\x T^*\O)^\O$ and let
$h\in\F(T^*\bar M\x \mathfrak{o}^*)$ be the function defined by
$\tilde h(\al_m,\al_g):=h(\al_m,\al_g g^{-1})$. Then the Hamiltonian vector field $X_{\tilde h}$ on $T^*\bar M\x T^*\O$, pushed
forward to $(T^*\bar M\x\mathfrak{o}^*)\times\O$ by the right
trivialization $\rho$ from \eqref{equiv_diffeo}, reads
\[
\left(\rho_*X_{\tilde h}\right)(\al_m,\si,g)=\left(\left(X_{h_\si}(\al_m),-\operatorname{ad}^*_{\frac{\delta h_{\al_m}}{\delta \si}}\si\right),\frac{\delta h_{\al_m}}{\delta \si}g\right),
\]
where $h_\si:T^*M \rightarrow \mathbb{R}  $ is defined by $h_\si( \alpha _m ):= h( \alpha _m , \sigma )$,  $h_{\al_m}: \mathfrak{o} ^\ast \rightarrow \mathbb{R}  $ is defined by $h_{ \alpha _m }( \sigma ):= h( \alpha _m , \sigma )$, and $X_{h _\sigma }$ denotes the Hamiltonian vector field associated to $h_ \sigma $ on $T^*\bar M$. 
\end{lemma} 

\begin{proposition}
The group $\overline\Aut_{\ham}(T^*\bar M\x T^*\O)$ acts infinitesimally transitively on the level sets of the momentum map $\J_R$ given in \eqref{doi}.
\end{proposition}
\begin{proof} Recall that the Lie algebra of $\overline\Aut_{\ham}(T^*\bar M\x T^*\O)$ consists of Hamiltonian vector fields $X_{\tilde h}$ associated to $ \mathcal{O} $-invariant Hamiltonian functions
$\tilde h \in \mathcal{F} (T^*\bar M \times T^* \mathcal{O} )^ \mathcal{O} $. 
Given $f=( \eta , \gamma )=( (\PP_\QQ,\si), \gamma  )\in Q_{KK}$, and using Lemma \ref{mara} above, the infinitesimal action reads
\begin{equation}\label{infinit_action} 
(X_{\tilde h})_{Q_{KK}}(\et,\ga)=\left(\left(X_{h_{\si}}\o\PP_\QQ,-\ad^*_{\frac{\de h}{\de\si}\o\et}\si\right),\left(\frac{\de h}{\de\si}\o\et\right)\ga\right).
\end{equation} 

We shall now write the derivative of the momentum map $\mathbf{J} _R$
\eqref{doi} at $( \eta , \gamma ) \in Q_{KK}$ in direction $(v_ \eta , u_ \gamma ) \in T_{(\eta , \gamma )} Q_{KK}$.
Note that we have $v_\et=(v_{\PP_\QQ},v_\si)$, where $v_{\PP_\QQ}:S\to T(T^*\bar M)$ is a smooth map covering $\PP_\QQ:S \rightarrow T^*\bar M$ and $v_\si:S\to\ou^*$ is a smooth map, and $u_\ga:S\to T\O$ is a smooth map covering $\ga:S \rightarrow \mathcal{O}$. In particular $u_\ga\ga^{-1}\in\F(S,\ou)$. Using the expression \eqref{deri} of the differential of the right logarithmic derivative map 
$\de^r:\F(S,\O)\mapsto \Om^1(S,\ou)$ shown in Lemma \ref{diff_log} below, we obtain the expression
\begin{multline*}
T_{(\et,\ga)}\J_R \cdot (v_\et,u_\ga)\\
=\left([\PP_\QQ^*\pounds _{v_{\PP_\QQ}}\Th_{\bar M}+\langle\ad^*_{u_\ga\ga^{-1}}\si+v_\si,\de^r\ga\rangle+\langle\si,\dd(u_\ga\ga^{-1})\rangle],\Ad^*_\ga(\ad^*_{u_\ga\ga^{-1}}\si+v_\si)\right)\\
\in\Om^1(S)/\dd\F(S)\x\F(S,\ou^*)=\mathfrak{aut}_{\vol}(S\x\O)^*.
\end{multline*}

In order to obtain the transitivity result, we have to show that any vector $(v_\et,u_\ga)$ in the kernel of $T_{(\et,\ga)}\J_R$ can be obtained from an infinitesimal generator
of the left $\overline\Aut_{\ham}(T^*\bar S\x T^*\O)$-action.
More precisely, for any vector $(v_\et,u_\ga)\in T_{( \eta , \gamma )}Q_{KK}$ with $\ad^*_{u_\ga\ga^{-1}}\si+v_\si=0$
and such that $\PP_\QQ^*\pounds _{v_{\PP_\QQ}}\Th_{\bar M}+\langle\si , \dd(u_\ga\ga^{-1})\rangle$ is an exact 1-form on $S$,
there exists $h\in\F(T^*\bar M\x\ou^*)$ such that  
$$v_{\PP_\QQ}=X_{ h_{\si}}\o\PP_\QQ,\quad 
v_\si=-\ad^*_{\frac{\de h}{\de\si}\o\et}\si,\quad u_\ga=\left( \frac{\de h}{\de\si}\o\et\right) \ga.$$

Define $j:=u_\ga\ga^{-1}\in\F(S,\ou)$, so that we have $u_\ga=j\ga$ and $v_\si=-\ad^*_j\si$. We will show that there exists 
$h\in\F(T^*\bar M\x\ou^*)$ such that 
\begin{equation}\label{to_be_shown} 
j= \frac{\delta h}{\delta \sigma }\circ \eta .
\end{equation} 
To achieve this we use the fact that $\PP_\QQ^*\pounds _{v_{\PP_\QQ}}\Th_{\bar M}+\langle\si,\dd j\rangle$ is exact, 
hence there exists $h_0\in\F(S)$ such that $\PP_\QQ^*\mathbf{i} _{v_{\PP_\QQ}}\dd\Th_{\bar M}+\langle\dd\si, j\rangle=\dd h_0$.
By a standard argument using a partition of unity,
we extend in a smooth way the function $h_0$ to a function 
$h_1\in\F(T^*\bar M\x \ou^*)$ via the embedding 
$ \eta :S \rightarrow T^*\bar M\x \ou^* $, i.e. such that $h_0=h_1\o\et$.
Then we consider the 1-form 
$$
\la:=(\dd h_1)\o\et-\mathbf{i} _{v_{\PP_\QQ}}(\dd\Th_{\bar M}\o\PP_\QQ)-j_\et
$$ 
on $T^*\bar M\x \ou^*$ along $S$, where $j_\et$ is the 1-form on $T^*\bar M\x\ou^*$ along $\et$
with first component zero and second component given by $j:S\to\ou={\ou^{**}}$. 

The form $\la$ vanishes on $T(\et(S))$, the tangent space to the submanifold $\et(S)$ of $T^*\bar M\x\ou^*$.
From 
Lemma \ref{function},
there exists a function  $h_2\in\F(T^*\bar M\x\ou^*)$ such that $\la =(\dd h_2)\o\et$.
The function $ h:=h_1-h_2$ satisfies $(\dd  h)\o\et=\mathbf{i} _{v_{\PP_\QQ}}(\dd\Th_{\bar M}\o\PP_\QQ)+j_\et$, therefore it verifies \eqref{to_be_shown} as desired.
 
From \eqref{to_be_shown} we obtain $u_\ga=(\frac{\de h}{\de\si}\o\et)\ga$ and $v_\si=-\ad^*_{\frac{\de h}{\de\si}\o\et}\si$.
Moreover, $(\dd h_{\si})\o\PP_\QQ=\mathbf{i} _{v_{\PP_\QQ}}(\dd\Th_{\bar M}\o\PP_\QQ)$,
hence $v_{\PP_\QQ}=X_{h_{\si}}\o\PP_\QQ$. Thus we have shown that any $(v_\eta , u _\gamma )$ in the kernel of $T_{( \eta , \gamma )} \mathbf{J} _R$ can be written as \eqref{infinit_action} for some $h:T^*\bar M\x \ou^* \rightarrow \mathbb{R} $. 
\end{proof}

\begin{lemma}\label{diff_log} 
The differential of the right logarithmic derivative
\[
\de^r:\F(M,\O)\to\Om^1(M,\ou),\quad \de^r\ga=(\dd \ga)\ga^{-1},
\]
at the point $\ga\in\F(M,\O)$ in direction $u_\ga\in\Ga(\ga^*T\O)$ is
\begin{equation}\label{deri}
\dd_\ga\de^r ( u_\ga)=\dd(u_\ga\ga^{-1})+\ad_{u_\ga\ga^{-1}}\de^r\ga.
\end{equation}
\end{lemma}

\begin{proof}
By taking the derivative of the identity $\de^r(\ga'\ga)=\de^r\ga'+\Ad_{\ga'}\de^r\ga$, relative to $ \gamma'$ at $\ga'=e$ in direction $\nu=u_\ga\ga^{-1}$, one obtains
\[
\dd_\ga\de^r ( u_\ga)=\dd_e\de^r (u_\ga\ga^{-1})+\ad_{u_\ga\ga^{-1}}\de^r\ga.
\]
The lemma follows now from the formula $\dd_e\de^r (\nu) =\dd \nu $ for all $\nu \in\F(M,\mathfrak{o} )$.
\end{proof}

\begin{proposition}
The group $\Aut_{\vol}(S\x\O)$ acts transitively on the level sets of the momentum map $\J_L$ given in  \eqref{unu}.
\end{proposition}

\begin{proof}
We need to show that  given two embeddings $(\et_1,\ga_1),(\et_2,\ga_2)\in\Emb_\O(S\x\O,T^*\bar M\x T^*\O)$ in the same level set of $\J_L$,
there exists a volume preserving automorphism $(\ps,b)\in\Aut_{\vol}(S\x\O)$, such that $(\et_2,\ga_2)=(\et_1\o\ps,(\ga_1\o\ps)b)$.

The equality $\J_L(\et_1,\ga_1)=\J_L(\et_2,\ga_2)$ reads
$\int_S(h\o\et_1)\mu_S=\int_S(h\o\et_2)\mu_S$, for all functions $h\in\F(T^*\bar S\x\ou^*)$,
so the embeddings $\et_1$ and $\et_2$ have the same image in $T^*\bar M\x\ou^*$. 
Therefore, there exists a unique diffeomorphism $\ps$ of $S$
such that $\et_2=\et_1\o\ps$.

It follows that the volume form $\nu_S:=(\ps^{-1})^*\mu_S$ satisfies 
$({\eta _{2}} )_*\mu_S=(\et_1)_*\nu_S$,
\ie $\int_S(h\o\et_1)\mu_S={\int_S(h\o\et_2)\mu_S} =\int_S(h\o\et_1)\nu_S$ 
for all $h\in\F(T^*\bar S\x\ou^*)$. 
Because $\et_1$ and $\et_2$ are embeddings, we conclude that $\mu_S=\nu_S$,
so $\ps^*\mu_S=\mu_S$ and $\ps$ is a volume preserving diffeomorphism of $S$. 

By defining the map $b:=(\ga_1^{-1}\o\ps)\ga_2\in\F(S,\O)$, 
we obtain that the automorphism $(\ps,b)\in\Aut_{\vol}(S\x\O)$
satisfies $(\et_2,\ga_2)=(\et_1\o\ps,(\ga_1\o\ps)b)=(\et_1,\ga_1)\cdot(\ps,b)$, as required.
\end{proof}

From the two transitivity results above, we obtain that the actions are mutually completely orthogonal and hence we get the following result.

\begin{theorem} The momentum maps \eqref{unu} and \eqref{doi} associated to the actions of the groups  $\overline\Aut_{\ham}(T^*\bar M\x T^*\O)\x_B\RR$ and $\Aut_{\vol}(S\x\O)$ on
$\Emb_{\O}(S\x\O,T^*{\bar M}\x T^*\mathcal{O})$, form a dual pair for the EPAut$_{\,vol}$ equations

\begin{picture}(80,100)(-70,0)%
\put(55,75){$\Emb_{\O}(S\x\O,T^*{\bar M}\x T^*\mathcal{O})$}

\put(90,50){$\mathbf{J}_L$}

\put(160,50){$\mathbf{J}_R$}

\put(40,15){$\mathcal{F}(T^*\bar M\x\mathfrak{o}^*)^*$
}

\put(160,15){$\mathfrak{aut}_{\vol}(S\x\O)^*$.
}

\put(130,70){\vector(-1, -1){40}}

\put(135,70){\vector(1,-1){40}}

\end{picture}
\end{theorem}
\paragraph{Conclusion and future works.} In this paper, we have shown that the pairs of momentum maps associated to the EPAut equations and its incompressible version arise from mutually completely orthogonal actions and are therefore dual pairs. We have obtained this result for the case of trivial principal bundles. For the incompressible situation we have restricted our study to the physical relevant case of the Yang-Mills phase space. Further studies will be necessary in order to show the same results without these restrictions.


{\footnotesize

{\footnotesize

\bibliographystyle{new}

\end{document}